\documentclass[12pt]{amsart}%
\usepackage[all]{xy}
\usepackage{amsmath}
\usepackage{graphicx}
\usepackage[small,nohug,heads=littlevee]{diagrams}
\usepackage{amscd}
\usepackage{geometry}
\usepackage{amsfonts}
\usepackage{amssymb}
\usepackage{bm}
\usepackage{rotating}%
\newtheorem{theorem}{Theorem}[section]
\theoremstyle{plain}

\newtheorem{claim}{Claim}

\newtheorem{example}{Example}

\newtheorem{lemma}[theorem]{Lemma}

\newtheorem{proposition}[theorem]{Proposition}
\newtheorem{remark}[theorem]{Remark}

\graphicspath{{converted_graphics/}}  	
\numberwithin{equation}{section}

\begin{document}
\title[Noncompact manifolds that are inward tame]{ Noncompact manifolds that are inward tame}
\author{C. R. Guilbault }
\address{Department of Mathematical Sciences, University of Wisconsin-Milwaukee,
Milwaukee, Wisconsin 53201}
\email{craigg@uwm.edu}
\author{F. C. Tinsley}
\address{Department of Mathematics, The Colorado College, Colorado Springs, Colorado 80903}
\email{ftinsley@coloradocollege.edu}
\thanks{Work on this project was aided by a Simons Foundation Collaboration Grant
awarded to the first author.}
\date{September 1, 2015}

\begin{abstract}
We continue our study of ends of non-compact manifolds, with a focus on the
inward tameness condition. For manifolds with compact boundary, inward
tameness, by itself, has significant implications. For example, such manifolds
have stable homology at infinity in all dimensions. We show that these
manifolds have `almost perfectly semistable' fundamental group at each of
their ends. That observation leads to further analysis of the group theoretic
conditions at infinity, and to the notion of a `near pseudo-collar' structure.
We obtain a complete characterization of $n$-manifolds ($n\geq6$) admitting
such a structure, thereby generalizing \cite{GT2}. We also construct examples
illustrating the necessity and usefulness of the new conditions introduced
here. Variations on the notion of a perfect group, with corresponding versions
of the Quillen Plus Construction, form an underlying theme of this work.

\end{abstract}
\maketitle

\section{Introduction}

In ~\cite{Gu1}, \cite{GT1} and \cite{GT2} we carried out a program to
generalize L.C. Siebenmann's famous Manifold Collaring Theorem \cite{Si} in
ways applicable to manifolds with non-stable fundamental group at infinity.
Motivated by some important examples of finite-dimensional manifolds and a
seminal paper by T.A. Chapman and Siebenmann \cite{CS} on Hilbert cube
manifolds, we chose the following definitions. \medskip

\noindent A manifold $N^{n}$ with compact boundary is called a \emph{homotopy
collar} if $\partial N^{n}\hookrightarrow N^{n}$ is a homotopy equivalence. If
$N^{n}$ contains arbitrarily small homotopy collar neighborhoods of infinity,
we call $N^{n}$ a \emph{pseudo-collar. }\medskip

\noindent Clearly, an actual open collar $N^{n}$, i.e., $N^{n}\approx\partial
N^{n}\times\lbrack0,\infty)$, is a special case of a pseudo-collar.
Fundamental to \cite{Si}, \cite{CS}, and our earlier work, is the notion of
`inward tameness'. \medskip

\noindent A manifold $M^{n}$ is \emph{inward tame }if each of its clean
neighborhoods of infinity is finitely dominated; it is \emph{absolutely inward
tame} if those neighborhoods all have finite homotopy type. \medskip

\noindent An alternative formulation of this definition (see
\S \ref{Subsection: Topology of ends of manifolds}) justifies the adjective
`inward'---a term that helps distinguish this version of tameness from a
similar, but inequivalent, version found elsewhere in the literature.

In \cite{GT2} a classification of pseudo-collarable $n$-manifolds for $6\leq
n<\infty$ was obtained. In simplified form, it says:\medskip

\begin{theorem}
[Pseudo-collarability Characterization---simple version]\label{PCT}\ \newline
A 1-ended $n$-manifold $M^{n}$ ($n\geq6$) with compact boundary is
pseudo-collarable iff the following conditions hold:

\begin{enumerate}
\item[a)] $M^{n}$ is absolutely inward tame, and

\item[b)] the fundamental group at infinity is $\mathcal{P}$%
-semistable.\medskip
\end{enumerate}
\end{theorem}

A `$\mathcal{P}$-semistable (or perfectly semistable) fundamental group at
infinity' indicates that an inverse sequence of fundamental groups of
neighborhoods of infinity can be arranged so that bonding homomorphisms are
surjective with perfect kernels.

By way of comparison, the simple version of Siebenmann's Collaring Theorem is
obtained by replacing b) with the stronger condition of $\pi_{1}$-stability,
while Chapman and Siebenmann's pseudo-collarability characterization for
Hilbert cube manifolds is obtained by omitting b) entirely. Thus, the
differences among these three results lie entirely in the fundamental group at infinity.

In this paper we take a close look at $n$-manifolds satisfying only the inward
tameness hypothesis. By necessity, our attention turns to the group theory at
the ends of those spaces. Unlike the case of infinite-dimensional manifolds,
CW complexes, or even $n$-manifolds with noncompact boundary, inward tameness
has major implications for the fundamental group at the ends of $n$-manifolds
with compact boundary. Unfortunately, inward tameness (ordinary or absolute)
does not imply $\mathcal{P}$-semistability---an example from \cite{GT1}
attests to that---but it comes remarkably close. One of the main results of
this paper is the following.

\begin{theorem}
\label{Theorem: inward tame implies AP-semistability}Let $M^{n}$ be an inward
tame $n$-manifold with compact boundary. Then $M^{n}$ has an $\mathcal{AP}%
$-semistable (almost perfectly semistable) fundamental group at each of its
finitely many ends.\medskip
\end{theorem}

Developing the appropriate group theory (including the definition of
$\mathcal{AP}$-semistable) and proving the above theorem are the initial goals
of this paper. After that is accomplished, we apply those investigations by
proving a structure theorem for manifolds that are inward tame, but not
necessarily pseudo-collarable.

\begin{theorem}
[Near Pseudo-collarability Characterization---simple version]%
\label{Theorem: NPC Characterization--simple version}A 1-ended $n$-manifold
$M^{n}$ ($n\geq6$) with compact boundary is nearly pseudo-collarable iff the
following conditions hold:

\begin{enumerate}
\item[a)] $M^{n}$ is absolutely inward tame, and

\item[b)] the fundamental group at infinity is $\mathcal{SAP}$%
-semistable.\medskip
\end{enumerate}
\end{theorem}

The notion of `near pseudo-collarability' will be defined and explored in
\S \ref{Section: Generalizing 1-sided cobordisms, homotopy collars, and pseudocollars}%
. For now, we note that nearly pseudo-collarable manifolds admit arbitrarily
small clean neighborhoods of infinity $N$, containing codimension 0
submanifolds $A$ for which $A\hookrightarrow N$ is a homotopy equivalence.
Obtaining a near pseudo-collar structure requires a slight strengthening of
$\mathcal{AP}$-semistability to $\mathcal{SAP}$-semistability (strong almost
perfect semistability). The essential nature of this stronger condition is
verified by a final result, in which our group-theoretic explorations come
together in a concrete set of examples.

\begin{theorem}
\label{Theorem: Existence of counterexamples}For all $n\geq6$, there exist
1-ended open $n$-manifolds that are absolutely inward tame but do not have
$\mathcal{SAP}$-semistable fundamental group at infinity, and thus, are not
nearly pseudo-collarable.\medskip
\end{theorem}

In \S \ref{Section: Remaining Questions}, we close with a discussion of some
open questions.

\begin{remark}
\emph{Throughout this paper attention is restricted to noncompact manifolds
with compact boundaries. When a boundary is noncompact, its end topology gets
entangled with that of the ambient manifold, leading to very different issues.
In the study of noncompact manifolds, a focus on those with compact boundaries
is analogous to a focus on closed manifolds in the study of compact manifolds.
An investigation of manifolds with noncompact boundaries is planned for
\cite{Gu2}.}
\end{remark}

\section{Definitions and Background}

\subsection{Variations on the notion of a perfect
group\label{Subsection: Variations on the notion of a perfect group}}

In this subsection we review the definition of `perfect group' and discuss
some variations.

Given elements $a$ and $b$ of a group $K$, the \emph{commutator }$a^{-1}%
b^{-1}ab$ will be denoted $\left[  a,b\right]  $. The \emph{commutator
subgroup} of $K$, denoted $\left[  K,K\right]  $ is the subgroup generated by
all commutators. It is a standard fact that $\left[  K,K\right]  $ is normal
in $K$ and is the smallest such subgroup with an abelian quotient. We call $K$
\emph{perfect} if $K=\left[  K,K\right]  $.

Now suppose $K$ and $J$ are normal subgroups of $G$. Define $\left[
K,J\right]  $ to be the subgroup of $G$ generated by the set of commutators
\[
\left[  k,j\right]  =\left\{  k^{-1}j^{-1}kj\mid k\in K\text{ and }j\in
J\right\}  .
\]
The following is standard and easy to verify.

\begin{lemma}
\label{Lemma: Generalized commutator subgroups}For normal subgroups $K$ and
$J$ of a group $G$,

\begin{enumerate}
\item $\left[  K,J\right]  \trianglelefteq G$,

\item $\left[  K,J\right]  \trianglelefteq K$ and $\left[  K,J\right]
\trianglelefteq J$, and

\item $\left[  K,J\right]  =\left[  J,K\right]  $.\medskip
\end{enumerate}
\end{lemma}

Given the above setup, we say that $K$ is $J$-\emph{perfect} if $K\subseteq
\left[  J,J\right]  $, and that $K$ is \emph{strongly }$J$-\emph{perfect} if
$K\subseteq\left[  K,J\right]  $\emph{.} By Lemma
\ref{Lemma: Generalized commutator subgroups}, both of these conditions imply
that $K\trianglelefteq J$; so we customarily begin with that as an assumption.

The following two Lemmas are immediate. We state them explicitly for the
purpose of comparison.

\begin{lemma}
Let $K\trianglelefteq J$ be normal subgroups of $G$.

\begin{enumerate}
\item $K$ is perfect if and only if each element of $K$ can be expressed as $%
{\textstyle\prod_{i=1}^{k}}
\left[  a_{i},b_{i}\right]  $ where $a_{i},b_{i}\in K$ for all $i$.

\item $K$ is $J$-perfect if and only if each element of $K$ can be expressed
as $%
{\textstyle\prod_{i=1}^{k}}
\left[  a_{i},b_{i}\right]  $ where $a_{i},b_{i}\in J$ for all $i$.

\item $K$ is strongly $J$-perfect if and only if each element of $K$ can be
expressed as \newline$%
{\textstyle\prod_{i=1}^{k}}
\left[  a_{i},b_{i}\right]  $ where $a_{i}\in K$ and $b_{i}\in J$ for all
$i$.\medskip
\end{enumerate}
\end{lemma}

\begin{lemma}
\label{basic implications}Let $K\trianglelefteq J\trianglelefteq L$ be normal
subgroups of $G$.

\begin{enumerate}
\item If $K$ is [strongly] $J$-perfect, then $K$ is [strongly] $L$-perfect for
every normal subgroup $L$ containing $J$.

\item $K$ is [strongly] $K$-perfect if and only if $K$ is a perfect group.
\end{enumerate}
\end{lemma}

\begin{remark}
\emph{Lemma ~\ref{basic implications} suggests a key theme: \textquotedblleft
The smaller the group }$L$ \emph{for which }$K$\emph{ is [strongly] }%
$L$-\emph{perfect, the closer }$K$\emph{ is to being a genuine perfect
group.\textquotedblright}
\end{remark}

The various levels of perfectness can be nicely characterized using
\emph{group homology}. The $%
\mathbb{Z}
$-homology of a group $G$ may be defined as the $%
\mathbb{Z}
$-homology of a $K\left(  G,1\right)  $ space $K_{G}$. If $\lambda
:G\rightarrow H$ is a homomorphism, there is a map $f_{\lambda}:K_{G}%
\rightarrow K_{H}$, unique up to basepoint preserving homotopy, inducing
$\lambda$ on fundamental groups. Define $\lambda_{\ast}:H_{\ast}\left(  G;%
\mathbb{Z}
\right)  \rightarrow H_{\ast}\left(  H;%
\mathbb{Z}
\right)  $ to be the homomorphisms induced by $f_{\lambda}$.

\begin{lemma}
\label{Lemma: group homology}Let $K\trianglelefteq J$, $i:K\hookrightarrow J$
be inclusion, and $q:J\rightarrow J/K$ be projection.

\begin{enumerate}
\item $K$ is perfect if and only if $H_{1}\left(  K;%
\mathbb{Z}
\right)  =0$.

\item $K$ is $J$-perfect if and only if $i_{\ast}:H_{1}\left(  K;%
\mathbb{Z}
\right)  \overset{0}{\rightarrow}H_{1}\left(  J;%
\mathbb{Z}
\right)  $ if and only if $q_{\ast}:H_{1}\left(  J;%
\mathbb{Z}
\right)  \overset{\cong}{\rightarrow}H_{1}\left(  J/K;%
\mathbb{Z}
\right)  .$

\item $K$ is strongly $J$-perfect if and only if $K$ is $J$-perfect and
$q_{\ast}:H_{2}\left(  J;%
\mathbb{Z}
\right)  \rightarrow H_{2}\left(  J/K;%
\mathbb{Z}
\right)  $ is surjective.\medskip
\end{enumerate}

\begin{proof}
Claim 1) is clear from the standard fact that $H_{1}\left(  K\right)  \cong
K/\left[  K,K\right]  $. Claim 2) can be verified with elementary group
theory. Claim 3) follows from a well-known 5-Term Exact Sequence\ due to
Stallings \cite{Stal} and Stammbach \cite{Stam}. Due to its importance in this
paper, we state it as a separate lemma.
\end{proof}
\end{lemma}

\begin{lemma}
[5-Term Exact Sequence for Group Homology]\label{Lemma: 5 term exact} Given a
normal subgroup $K$ of a group $J$, there is a natural exact sequence:%
\[
H_{2}\left(  J;%
\mathbb{Z}
\right)  \rightarrow H_{2}\left(  J/K;%
\mathbb{Z}
\right)  \rightarrow K/\left[  K,J\right]  \rightarrow H_{1}\left(  J;%
\mathbb{Z}
\right)  \rightarrow H_{1}\left(  J/K;%
\mathbb{Z}
\right)  \rightarrow0\text{.}%
\]

\end{lemma}

The following elementary facts about group homology will be useful.

\begin{lemma}
\label{Lemma: induced surjections on H_2}Let $f:X\rightarrow Y$ be a map
between connected CW complexes and $\lambda:\pi_{1}\left(  X\right)
\rightarrow\pi_{1}\left(  Y\right)  $ the induced homomorphism. Then

\begin{enumerate}
\item $H_{1}\left(  X;\mathbb{Z}\right)  \cong H_{1}\left(  \pi_{1}\left(
X,*\right)  ;\mathbb{Z}\right)  $,

\item $f_{\ast}:H_{1}\left(  X;%
\mathbb{Z}
\right)  \rightarrow H_{1}\left(  Y;%
\mathbb{Z}
\right)  $ realizes $\lambda_{\ast}:H_{1}\left(  \pi_{1}\left(  X\right)  ;%
\mathbb{Z}
\right)  \rightarrow H_{1}\left(  \pi_{1}\left(  Y\right)  ;%
\mathbb{Z}
\right)  $, and

\item if $f_{\ast}:H_{2}\left(  X;%
\mathbb{Z}
\right)  \rightarrow H_{2}(Y;%
\mathbb{Z}
)$ is surjective, then $\lambda_{\ast}:H_{2}\left(  \pi_{1}\left(  X\right)  ;%
\mathbb{Z}
\right)  \rightarrow\allowbreak H_{2}\left(  \pi_{1}\left(  Y\right)  ;%
\mathbb{Z}
\right)  $ is also surjective.
\end{enumerate}
\end{lemma}

\begin{proof}
Build a $K\left(  \pi_{1}\left(  X\right)  ,1\right)  $ complex $X^{\prime}$
by attaching cells of dimension $\geq3$ to $X$ and a $K\left(  \pi_{1}\left(
Y\right)  ,1\right)  $ complex $Y^{\prime}$ by attaching cells of dimension
$\geq3$ to $Y$. Both $X\overset{i}{\hookrightarrow}X^{\prime}$ and
$Y\overset{j}{\hookrightarrow}Y^{\prime}$ induce isomorphisms on $\pi_{1}$ and
$H_{1}$, so (1) follows. Use the asphericity of $Y^{\prime}$ to extend $f$ to
$f^{\prime}:X^{\prime}\rightarrow Y^{\prime}$, also inducing $\lambda$ on
$\pi_{1}$. Clearly $i_{\ast}:H_{2}\left(  X;%
\mathbb{Z}
\right)  \rightarrow H_{2}\left(  X^{\prime};%
\mathbb{Z}
\right)  $ and $j_{\ast}:H_{2}\left(  Y;%
\mathbb{Z}
\right)  \rightarrow H_{2}\left(  Y^{\prime};%
\mathbb{Z}
\right)  $ are surjective.

This gives a commutative diagram%
\[%
\begin{array}
[c]{ccc}%
H_{2}\left(  X;%
\mathbb{Z}
\right)  & \overset{f_{\ast}}{\twoheadrightarrow} & H_{2}\left(  Y;%
\mathbb{Z}
\right) \\
i_{\ast}\downarrow &  & \downarrow j_{\ast}\\
H_{2}\left(  \pi_{1}\left(  X\right)  ;%
\mathbb{Z}
\right)  & \overset{f_{\ast}^{\prime}}{\longrightarrow} & H_{2}\left(  \pi
_{1}\left(  Y\right)  ;%
\mathbb{Z}
\right)
\end{array}
\]
Since the other maps are all surjective, so is $f_{\ast}^{\prime}$.\bigskip
\end{proof}

Lastly we offer a topological characterization of the various levels of
perfectness. For the purposes of this paper, these are possibly the most useful.

Let $S_{g}$ denote a compact orientable surface of genus $g$ with a single
boundary component. A collection of oriented simple closed curves $\left\{
\alpha_{1},\beta_{1},\alpha_{2},\beta_{2},\cdots,\alpha_{g},\beta_{g}\right\}
$ on $S_{g}$ with the property that each $\alpha_{i}$ intersects $\beta_{i}$
transversely at a single point, and each of $\alpha_{i}\cap\alpha_{j}$,
$\beta_{i}\cap\beta_{j}$, and $\alpha_{i}\cap\beta_{j}$ is empty when $i\neq
j$, is called a \emph{complete set of handle curves} for $S_{g}$. A complete
set of handle curves on $S_{g}$ is not unique; however given any such set,
there exists a homeomorphism of $S_{g}$ to the `disk with $g$ handles'
pictured in Figure \ref{Fig1-handle curves} taking each $\alpha_{i}$ and
$\beta_{i}$ to the corresponding curves in the diagram.%
\begin{figure}
[ptb]
\begin{center}
\includegraphics[
height=1.785in,
width=3.4506in
]%
{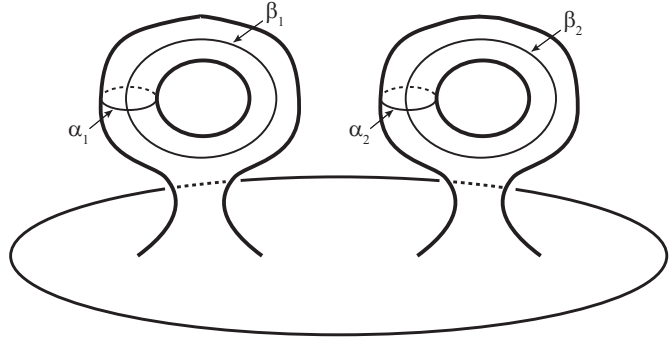}%
\caption{Complete set of handle curves ($g=2$ case)}%
\label{Fig1-handle curves}%
\end{center}
\end{figure}

Given a (not necessarily embedded) loop $\gamma$ in a topological space $X$,
we say that $\gamma$ \emph{bounds a compact orientable surface in} $X$ if, for
some $g$, there exists a map $f:S_{g}\rightarrow X$ such that $\left.
f\right\vert _{\partial S_{g}}=\gamma$. Notice that we do not require that $f$
be an embedding. We often abuse terminology slightly by saying that $\gamma$
bounds the surface $S_{g}$ in $X$. Similarly, we often do not distinguish
between a set of handle curves on $S_{g}$ and their images in $X$.

\begin{lemma}
\label{Lemma: geometric perfect} Let $X$ be a space with $\pi_{1}\left(
X,x_{0}\right)  \cong G$ and let $K\trianglelefteq J$ be normal subgroups of
$G$. Then

\begin{enumerate}
\item $K$ is perfect if and only if each loop $\gamma$ in $X$ representing an
element of $K$ bounds a surface $S_{g}$ in $X$ containing a complete set of
handle curves $\left\{  \alpha_{1},\beta_{1},\alpha_{2},\beta_{2}%
,\allowbreak\cdots\allowbreak,\alpha_{g},\beta_{g}\right\}  $ with each
$\alpha_{i}$ and $\beta_{i}$ belonging to $K$.

\item $K$ is $J$-perfect if and only if each loop $\gamma$ in $X$ representing
an element of $K$ bounds a surface $S_{g}$ in $X$ containing a complete set of
handle curves $\left\{  \alpha_{1},\beta_{1},\alpha_{2},\beta_{2}%
,\allowbreak\cdots\allowbreak,\alpha_{g},\beta_{g}\right\}  $ with each
$\alpha_{i}$ and $\beta_{i}$ belonging to $J$.

\item $K$ is strongly $J$-perfect if and only if each loop $\gamma$ in $X$
representing an element of $K$ bounds a surface $S_{g}$ in $X$ containing a
complete set of handle curves \newline$\left\{  \alpha_{1},\beta_{1}%
,\alpha_{2},\beta_{2},\allowbreak\cdots\allowbreak,\alpha_{g},\beta
_{g}\right\}  $ with each $\alpha_{i}$ belonging to $K$ and each $\beta_{i}$
belonging to $J$.\bigskip
\end{enumerate}
\end{lemma}

\begin{remark}
\emph{We are being informal in the statement of Lemma
\ref{Lemma: geometric perfect}. Since the handle curves are not based, we
should also choose, for each pair }$\left(  \alpha_{i},\beta_{i}\right)
$\emph{, an arc }$\tau_{i}$\emph{ in }$S_{g}$\emph{ from }$x_{0}$\emph{ to
}$p_{i}=\alpha_{i}\cap\beta_{i}$\emph{. The element of }$\pi_{1}\left(
X,x_{0}\right)  $\emph{ represented by }$\alpha_{i}$\emph{ is then }$\tau
_{i}\ast\alpha_{i}\ast\tau_{i}^{-1}$\emph{, and similarly for }$\beta_{i}%
$\emph{. Notice that, by normality, the question of whether one of these loops
belongs to }$K$\emph{ or }$J$\emph{ is independent of the choice of }$\tau
_{i}$\emph{.}
\end{remark}

\subsection{Algebra of inverse sequences}

Understanding the `fundamental group at infinity' requires the language of
inverse sequences. We briefly review the necessary definitions and terminology.

Throughout this subsection all arrows denote homomorphisms, while those of
type $\twoheadrightarrow$ or $\twoheadleftarrow$ specify surjections. The
symbol $\cong$ denotes isomorphisms.

Let
\[
G_{0}\overset{\lambda_{1}}{\longleftarrow}G_{1}\overset{\lambda_{2}%
}{\longleftarrow}G_{2}\overset{\lambda_{3}}{\longleftarrow}\cdots
\]
be an inverse sequence of groups and homomorphisms. A \emph{subsequence} is an
inverse sequence of the form
\[
G_{i_{0}}\overset{\lambda_{i_{0}+1}\circ\cdots\circ\lambda_{i_{1}}%
}{\longleftarrow}G_{i_{1}}\overset{\lambda_{i_{1}+1}\circ\cdots\circ
\lambda_{i_{2}}}{\longleftarrow}G_{i_{2}}\overset{\lambda_{i_{2}+1}\circ
\cdots\circ\lambda_{i_{3}}}{\longleftarrow}\cdots.
\]
In the future we denote a composition $\lambda_{i}\circ\cdots\circ\lambda_{j}$
($i\leq j$) by $\lambda_{i,j}$.

Sequences $\left\{  G_{i},\lambda_{i}\right\}  $ and $\left\{  H_{i},\mu
_{i}\right\}  $ are \emph{pro-isomorphic} if, after passing to subsequences,
there exists a commuting diagram:
\[%
\begin{array}
[c]{ccccccc}%
G_{i_{0}} & \overset{\lambda_{i_{0}+1,i_{1}}}{\longleftarrow} & G_{i_{1}} &
\overset{\lambda_{i_{1}+1,i_{2}}}{\longleftarrow} & G_{i_{2}} & \overset
{\lambda_{i_{2}+1,i_{3}}}{\longleftarrow} & \cdots\\
& \nwarrow\quad\swarrow &  & \nwarrow\quad\swarrow &  & \nwarrow\quad\swarrow
& \\
& H_{j_{0}} & \overset{\mu_{j_{0}+1,j_{1}}}{\longleftarrow} & H_{j_{1}} &
\overset{\mu_{j_{1}+1,j_{2}}}{\longleftarrow} & H_{j_{2}} & \cdots
\end{array}
.
\]
Clearly an inverse sequence is pro-isomorphic to each of its subsequences. To
avoid tedious notation, we often do not distinguish $\left\{  G_{i}%
,\lambda_{i}\right\}  $ from its subsequences. Instead we assume $\left\{
G_{i},\lambda_{i}\right\}  $ has the properties of a preferred
subsequence---prefaced by the words `after passing to a subsequence and relabeling'.

An inverse sequence $\left\{  G_{i},\lambda_{i}\right\}  $ is \emph{stable} if
it is pro-isomorphic to an $\left\{  H_{i},\mu_{i}\right\}  $ for which each
$\mu_{i}$ is an isomorphism. A more usable formulation is that $\left\{
G_{i},\lambda_{i}\right\}  $ is stable if, after passing to a subsequence and
relabeling, there is a commutative diagram of the form
\begin{equation}%
\begin{array}
[c]{ccccccccc}%
G_{0} & \overset{\lambda_{1}}{\longleftarrow} & G_{1} & \overset{\lambda_{2}%
}{\longleftarrow} & G_{2} & \overset{\lambda_{3}}{\longleftarrow} & G_{3} &
\overset{\lambda_{4}}{\longleftarrow} & \cdots\\
& \nwarrow\quad\swarrow &  & \nwarrow\quad\swarrow &  & \nwarrow\quad\swarrow
&  &  & \\
& im(\lambda_{1}) & \overset{\cong}{\longleftarrow} & im(\lambda_{2}) &
\overset{\cong}{\longleftarrow} & im(\lambda_{3}) & \overset{\cong%
}{\longleftarrow} & \cdots &
\end{array}
\tag{$\ast$}%
\end{equation}
where all unlabeled maps are obtained by restriction. If $\left\{  H_{i}%
,\mu_{i}\right\}  $ can be chosen so that each $\mu_{i}$ is an epimorphism, we
call our sequence \emph{semistable }(or \emph{Mittag-Leffler, }%
or\emph{\ pro-epimorphic}). In that case, it can be arranged that the maps in
the bottom row of ($\ast$) are epimorphisms. Similarly, if $\left\{  H_{i}%
,\mu_{i}\right\}  $ can be chosen so that each $\mu_{i}$ is a monomorphism, we
call our sequence \emph{pro-monomorphic}; it can then be arranged that the
restriction maps in the bottom row of ($\ast$) are monomorphisms. It is easy
to show that an inverse sequence that is semistable and pro-monomorphic is stable.

An inverse sequence is \emph{perfectly semistable} if it is pro-isomorphic to
an inverse sequence
\[
G_{0}\overset{\lambda_{1}}{\twoheadleftarrow}G_{1}\overset{\lambda_{2}%
}{\twoheadleftarrow}G_{2}\overset{\lambda_{3}}{\twoheadleftarrow}\cdots
\]
of finitely presentable groups and surjections where each $\ker\left(
\lambda_{i}\right)  $ is perfect. A straightforward argument \cite[Cor.
1]{Gu1} shows that sequences of this type behave well under passage to subsequences.

\subsection{Augmented inverse sequences and almost perfect semistability.}

An \emph{augmentation} of an inverse sequence $\left\{  G_{i},\lambda
_{i}\right\}  $ is a sequence $\left\{  L_{i}\right\}  $, where $L_{i}%
\trianglelefteq G_{i}$ and $\lambda_{i}\left(  L_{i}\right)  \leq L_{i-1}$ for
each $i$. The corresponding \emph{augmentation sequence} is the sequence
$\left\{  L_{i},\left.  \lambda\right\vert _{L_{i}}\right\}  $.

The \emph{minimal augmentation} (or the \emph{unaugmented case}) occurs when
$L_{i}=\left\{  1\right\}  $; the \emph{maximal augmentation} is the case
where $L_{i}=G_{i}$; and the \emph{standard augmentation} occurs when
$L_{i}=\ker\lambda_{i}$ for each $i$. Any augmentation where $L_{i}\leq
\ker\lambda_{i}$ for each $i$ is called a \emph{small augmentation. }For each
subsequence $\left\{  G_{k_{i}}\right\}  $ of a sequence $\left\{
G_{i},\lambda_{i}\right\}  $ augmented by $\left\{  L_{i}\right\}  $, there is
a corresponding augmentation $\left\{  L_{k_{i}}\right\}  $.

We say that $\left\{  G_{i},\lambda_{i}\right\}  $ satisfies the $\left\{
L_{i}\right\}  $\emph{-perfectness property} if, for each $i$, $\ker
\lambda_{i}$ is $\lambda_{i}^{-1}\left(  L_{i-1}\right)  $-perfect; it
satisfies the \emph{strong }$\left\{  L_{i}\right\}  $\emph{-perfectness
property} if each $\ker\lambda_{i}$ is strongly $\lambda_{i}^{-1}\left(
L_{i-1}\right)  $-perfect. More concisely, if $K_{i}=\ker\lambda_{i}$ and
$J_{i}=\lambda_{i}^{-1}\left(  L_{i-1}\right)  $, these conditions require
that each $K_{i}$ be [strongly] $J_{i}$-perfect.

Employing the above terminology, we can restate the definition perfect
semistability (abbreviated $\mathcal{P}$\emph{-semistable}) by requiring that
the sequence be pro-isomorphic an inverse sequence of finitely presented
groups and surjections satisfying the $\left\{  L_{i}\right\}  $-perfectness
property for the minimal augmentation $\left\{  L_{i}\right\}  =\left\{
1\right\}  $. More generally, we call an inverse sequence of groups:

\begin{itemize}
\item $\mathcal{AP}$\emph{-semistable} (for almost perfectly semistable) if it
is pro-isomorphic to an inverse sequence $\left\{  G_{i},\lambda_{i}\right\}
$ of finitely presentable groups and surjections, satisfying the $\left\{
L_{i}\right\}  $-perfectness property for some small augmentation $\left\{
L_{i}\right\}  $, and

\item $\mathcal{SAP}$\emph{-semistable} (for strongly almost perfectly
semistable) if it is pro-isomorphic to an inverse sequence $\left\{
G_{i},\lambda_{i}\right\}  $ of finitely presentable groups and surjections
satisfying the strong $\left\{  L_{i}\right\}  $-perfectness property for some
small augmentation $\left\{  L_{i}\right\}  $.
\end{itemize}

\begin{remark}
\emph{Note that an inverse sequence satisfies the [strong] }$\left\{
L_{i}\right\}  $\emph{-perfectness property for }some\emph{ small augmentation
}$\left\{  L_{i}\right\}  $\emph{ if and only if it satisfies that property
for the standard augmentation.}
\end{remark}

When applying sequences of the above types to geometric constructions, it is
frequently desirable to pass to subsequences without losing the defining
property of the sequence. For that reason, the following observation is crucial.

\begin{proposition}
\label{Prop: nearly sub}If an inverse sequence $\left\{  G_{i},\lambda
_{i}\right\}  $ of surjections augmented by $\left\{  L_{i}\right\}  $
satisfies the [strong] $\left\{  L_{i}\right\}  $-perfectness property, then
any subsequence $\left\{  G_{k_{i}}\right\}  $ satisfies the corresponding
[strong] $\left\{  L_{k_{i}}\right\}  $-perfectness property.
\end{proposition}

\begin{proof}
Since the proofs for perfectness and strong perfectness are similar, we prove
only the latter. Assume $\left\{  G_{i},\lambda_{i}\right\}  $ augmented by
$\left\{  L_{i}\right\}  $ satisfies strong $\left\{  L_{i}\right\}
$-perfectness.
Simplifying notation, a portion of the given subsequence becomes%
\[
G_{a}\overset{\lambda_{a+1,b}}{\longleftarrow}G_{b}\overset{\lambda_{b+1,c}%
}{\longleftarrow}G_{c}%
\]
where $-1\leq a<b<c$. We must show that $\ker\left(  \lambda_{b+1,c}\right)
\subseteq\left[  \ker\left(  \lambda_{b+1,c}\right)  ,\lambda_{b+1,c}%
^{-1}\left(  L_{b}\right)  \right]  $.

Suppose the proposition holds for $j<c$. If $c=b+1$, then $\lambda
_{b+1,c}=\lambda_{c}$, and the result follows by hypothesis.
Now, assume $c\geq b+2$ and write
\[
\lambda_{b+1,c}=\lambda_{b+1,c-1}\circ\lambda_{c}:G_{c}\rightarrow
G_{c-1}\rightarrow G_{b}\text{.}%
\]
Let $\omega\in\ker\left(  \lambda_{b+1,c}\right)  $; then $\lambda_{c}\left(
\omega\right)  \in\ker\left(  \lambda_{b+1,c-1}\right)  $. By induction,
$\ker\left(  \lambda_{b+1,c-1}\right)  \subseteq$\newline$\left[  \ker\left(
\lambda_{b+1,c-1}\right)  ,\lambda_{b+1,c-1}^{-1}\left(  L_{b}\right)
\right]  $; so, $\lambda_{c}(\omega)$ is a product of commutators $[\alpha
_{m},\beta_{m}]$ where $\beta_{m}\in\lambda_{b+1,c-1}^{-1}\left(
L_{b}\right)  $ and
$\alpha_{m}^{\text{\ }}\in\allowbreak\ker\left(  \lambda_{b+1,c-1}\right)  $.
Since $\lambda_{c}$ is surjective over $G_{c-1}$ we identify for each $m$ a
pair of elements $\alpha_{m}^{\,\prime},\beta_{m}^{\,\prime}\in G_{c}$ that
map to $\alpha_{m}$ and $\beta_{m}$, respectively. Thus, $\beta_{m}^{\,\prime
}\in\lambda_{b+1,c}^{-1}\left(  L_{b}\right)  $,
$\alpha_{m}^{\,\prime}\in\ker\left(  \lambda_{b+1,c}\right)  $, and
$[\alpha_{m}^{\,\prime},\beta_{m}^{\,\prime}]\in\lbrack\ker\left(
\lambda_{b+1,c}\right)  ,\lambda_{b+1,c}^{-1}\left(  L_{b}\right)  ]$.

Now, let $\nu$ be the product of the commutators with $\left[  \alpha
_{m}^{\,\prime},\beta_{m}^{\,\prime}\right]  $ replacing $\left[  \alpha
_{m},\beta_{m}\right]  $. By construction, $\lambda_{c}(\omega)=\lambda
_{c}(\nu)$ and
$\nu\in\lbrack\ker\left(  \lambda_{b+1,c}\right)  ,\lambda_{b+1,c}^{-1}\left(
L_{b}\right)  ]$. Thus,%
\[
\omega v^{-1}\in\ker\left(  \lambda_{c}\right)  \subseteq\left[  \ker\left(
\lambda_{c}\right)  ,\lambda_{c}^{-1}\left(  L_{c-1}\right)  \right]
\subseteq\left[  \ker\left(  \lambda_{b+1,c}\right)  ,\lambda_{b+1,c}%
^{-1}\left(  L_{b}\right)  \right]  \text{.}%
\]
Consequently,
$\omega\in\lbrack\ker\left(  \lambda_{b+1,c}\right)  ,\lambda_{b+1,c}%
^{-1}\left(  L_{b}\right)  ]$ as well, completing the proof of the proposition.
\end{proof}

\subsection{Topology of ends of
manifolds\label{Subsection: Topology of ends of manifolds}}

Next we supply some topological definitions and background. Throughout the
paper, $\approx$ represents homeomorphism and $\simeq$ indicates homotopic
maps or homotopy equivalent spaces. The word \emph{manifold} means
\emph{manifold with (possibly empty) boundary}. A manifold is \emph{open} if
it is non-compact and has no boundary. As noted earlier, we restrict our
attention to manifolds with compact boundaries.

For convenience, all manifolds are assumed to be PL; analogous results may be
obtained for smooth or topological manifolds in the usual ways. Our standard
resource for PL topology is \cite{RS}. Some of the results presented here are
valid in all dimensions. Others are valid in dimensions $\geq4$ or $\geq5$,
but require the purely topological $4$-dimensional techniques found in
\cite{FQ} for the $4$ and/or $5$ dimensional cases; there the conclusions are
only topological. The main focus of this paper is on dimensions $\geq6$.

Let $M^{n}$ be a manifold with compact (possibly empty) boundary. A set
$N\subseteq M^{n}$ is a \emph{neighborhood of infinity} if $\overline{M^{n}%
-N}$ is compact. A neighborhood of infinity $N$ is \emph{clean} if

\begin{itemize}
\item $N$ is a closed subset of $M^{n}$,

\item $N\cap\partial M^{n}=\emptyset$, and

\item $N$ is a codimension 0 submanifold of $M^{n}$ with bicollared boundary.
\end{itemize}

\noindent It is easy to see that each neighborhood of infinity contains a
clean neighborhood of infinity.

We say that $M^{n}$ \emph{has }$k$ \emph{ends }if it contains a compactum $C$
such that, for every compactum $D$ with $C\subseteq D$, $M^{n}-D$ has exactly
$k$ unbounded components, i.e., $k$ components with noncompact closures. When
$k$ exists, it is uniquely determined; if $k$ does not exist, we say $M^{n}$
\emph{has infinitely many ends}. If $M^{n}$ is $k$-ended, then it contains a
clean neighborhood of infinity $N$ consisting of $k$ connected components,
each of which is a 1-ended manifold with compact boundary. Thus, when studying
manifolds with finitely many ends, it suffices to understand the $\emph{1}%
$\emph{-ended} situation. That is the case in this paper, where our standard
hypotheses ensure finitely many ends. (See Theorem ~\ref{semistable}.)

A connected clean neighborhood of infinity with connected boundary is called a
\emph{0-neigh\-bor\-hood of infinity}. A $0$-neighborhood of infinity $N$ for
which $\partial N\hookrightarrow N$ induces a $\pi_{1}$-isomorphism is called
a \emph{generalized 1-neighborhood of infinity}. If, in addition, $\pi
_{j}\left(  N,\partial N\right)  =0$ for $j\leq k$, then $N$ is a
\emph{generalized k-neighborhood of infinity}.

A nested sequence $N_{0}\supset N_{1}\supset N_{2}\supset\cdots$ of
neighborhoods of infinity is \emph{cofinal }if $\bigcap_{i=0}^{\infty}%
N_{i}=\emptyset$. We will refer to any cofinal sequence $\left\{
N_{i}\right\}  $ of closed neighborhoods of infinity with $N_{i+1}\subseteq
int\left(  N_{i}\right)  $, for all $i$, as an \emph{end structure }for
$M^{n}$. Descriptors will be added to indicate end structures with additional
properties. For example, if each $N_{i}$ is clean we call $\left\{
N_{i}\right\}  $ a \emph{clean end structure}; if each $N_{i}$ is clean and
connected we call $\left\{  N_{i}\right\}  $ a \emph{clean connected end
structure}; and if each $N_{i}$ is a generalized $k$-neighborhood of infinity,
we call $\left\{  N_{i}\right\}  $ a \emph{generalized k-neighborhood end
structure}.

\begin{remark}
\emph{The word `generalized' in the above definitions is in deference to
Siebenmann's terminology in \cite{Si} where the ambient manifold }$M^{n}%
$\emph{ is assumed to have stable fundamental group at infinity. There a
(non-generalized) k-neighborhood of infinity\ }$N$\emph{ is also required to
satisfy }$\pi_{1}\left(  \varepsilon\left(  M^{n}\right)  \right)
\overset{\cong}{\longrightarrow}\pi_{1}\left(  N\right)  $\emph{.}
\end{remark}

Building upon the above terminology, the primary goal of this paper can be
described as: Identify, construct, and detect the existence of various end
structures for manifolds. A central example---the \emph{pseudo-collar }can be
described as an end structure $\left\{  N_{i}\right\}  $ where each $N_{i}$ is
a homotopy collar.

We say $M^{n\text{ }}$is \emph{inward tame} if, for arbitrarily small
neighborhoods of infinity $N$, there exist homotopies $H:N\times\left[
0,1\right]  \rightarrow N$ such that $H_{0}=\operatorname*{id}{}_{N}$ and
$\overline{H_{1}\left(  N\right)  }$ is compact. Thus inward tameness means
each neighborhood of infinity can be pulled into a compact subset of itself.
In this case we refer to $H$ as a \emph{taming homotopy}.

In \cite{Gu1}, the existence of generalized $\left(  n-3\right)
$-neighborhood end structures is shown for all inward tame $M^{n}$ ($n\geq5$).

Recall that a space $X$ is \emph{finitely dominated} if there exists a finite
complex $K$ and maps $u:X\rightarrow K$ and $d:K\rightarrow X$ such that
$d\circ u\simeq\operatorname*{id}{}_{X}$. The following lemma uses this notion
to offer equivalent formulations of inward tameness.

\begin{lemma}
\label{Lemma: inward tame} (See \cite[Lemma 2.4]{GT1}) For a manifold $M^{n}$,
the following are equivalent.

\begin{enumerate}
\item $M^{n}$ is inward tame.

\item Each clean neighborhood of infinity in $M^{n}$ is finitely dominated.

\item For each clean end structure $\left\{  N_{i}\right\}  $, the inverse
sequence
\[
N_{0}\overset{j_{1}}{\hookleftarrow}N_{1}\overset{j_{2}}{\hookleftarrow}%
N_{2}\overset{j_{3}}{\hookleftarrow}\cdots
\]
is pro-homotopy equivalent to an inverse sequence of finite polyhedra.
\end{enumerate}
\end{lemma}

Given a clean connected end structure $\left\{  N_{i}\right\}  _{i=0}^{\infty
}$, base points $p_{i}\in N_{i}$, and paths $\alpha_{i}\subseteq N_{i}$
connecting $p_{i}$ to $p_{i+1}$, we obtain an inverse sequence:
\[
\pi_{1}\left(  N_{0},p_{0}\right)  \overset{\lambda_{1}}{\longleftarrow}%
\pi_{1}\left(  N_{1},p_{1}\right)  \overset{\lambda_{2}}{\longleftarrow}%
\pi_{1}\left(  N_{2},p_{2}\right)  \overset{\lambda_{3}}{\longleftarrow}%
\cdots.
\]
\smallskip Here, each $\lambda_{i+1}:\pi_{1}\left(  N_{i+1},p_{i+1}\right)
\rightarrow\pi_{1}\left(  N_{i},p_{i}\right)  $ is the homomorphism induced by
inclusion followed by the change of base point isomorphism determined by
$\alpha_{i}$. The singular ray obtained by piecing together the $\alpha_{i}$'s
is called the \emph{base ray }for the inverse sequence. Provided the sequence
is semistable, its pro-isomorphism class does not depend on any of the choices
made above (see \cite{Gu3} or \cite[\S 16.2]{Ge}). In the absence of
semistability, the pro-isomorphism class of the inverse sequence depends on
the base ray; hence, the ray becomes part of the data. The same procedure may
be used to define $\pi_{k}\left(  \varepsilon\left(  M^{n}\right)  \right)  $
for all $k\geq1$. Similarly, but without need for a base ray or connectedness,
we may define $H_{k}\left(  \varepsilon\left(  M^{n}\right)  \right)  $.

In \cite{Wa}, Wall showed that each finitely dominated connected space $X$
determines a well-defined $\sigma\left(  X\right)  \in\widetilde{K}_{0}\left(
%
\mathbb{Z}
\left[  \pi_{1}X\right]  \right)  $ (the reduced projective class group) that
vanishes if and only if $X$ has the homotopy type of a finite complex. Given a
clean connected end structure $\left\{  N_{i}\right\}  _{i=0}^{\infty}$ for an
inward tame $M^{n}$, we have a Wall finiteness obstruction $\sigma(N_{i})$ for
each $i$. These may be combined into a single obstruction
\[
\sigma_{\infty}(M^{n})=\left(  -1\right)  ^{n}(\sigma(N_{0}),\sigma
(N_{1}),\sigma(N_{2}),\cdots)\in\widetilde{K}_{0}\left(  \pi_{1}\left(
\varepsilon\left(  M^{n}\right)  \right)  \right)  \equiv\underleftarrow{\lim
}\widetilde{K}_{0}\left(
\mathbb{Z}
\left[  \pi_{1}N_{i}\right]  \right)
\]
that is well-defined and which vanishes if and only if each clean neighborhood
of infinity in $M^{n}$ has finite homotopy type. See \cite{CS} or \cite{Gu1}
for details.

We may now state the full version of the main theorem of \cite{GT2}.

\begin{theorem}
[Pseudo-collarability Characterization---complete version]%
\label{PCT technical version}A 1-ended $n$-manifold $M^{n}$ ($n\geq6$) with
compact boundary is pseudo-collarable if and only if the following conditions hold:

\begin{enumerate}
\item $M^{n}$ is inward tame,

\item $\pi_{1}(\varepsilon(M^{n}))$ is $\mathcal{P}$-semistable, and

\item $\sigma_{\infty}\left(  M^{n}\right)  =0\in\widetilde{K}_{0}\left(
\pi_{1}\left(  \varepsilon(M^{n}\right)  \right)  )$.\bigskip
\end{enumerate}
\end{theorem}

\section{Some consequences of inward tameness\label{consequences}}

In this section we show that, for manifolds with compact boundary, the inward
tameness condition, by itself, has significant implications. The main goal is
a proof of Theorem \ref{Theorem: inward tame implies AP-semistability}---that
every inward tame manifold with compact boundary has $\mathcal{AP}$-semistable
fundamental group at each of its finitely many ends. Results in this section
are valid in all (finite) dimensions.

Begin by recalling a theorem from \cite{GT1}.

\begin{theorem}
\label{semistable}If an $n$-manifold with compact (possibly empty) boundary is
inward tame, then it has finitely many ends, each of which has semistable
fundamental group and stable homology in all dimensions.\medskip
\end{theorem}

\begin{remark}
\emph{Note that }none\emph{ of the above conclusions is valid for Hilbert cube
manifolds, polyhedra, or manifolds with noncompact boundary. See, for example,
\cite[\S 4.5]{Gu3}.}
\end{remark}

As preparation for the proof of Theorem
\ref{Theorem: inward tame implies AP-semistability}, we look at an easier
result that follows directly from Theorem \ref{semistable}.

Let $M^{n}$ be an inward tame $n$-manifold with compact boundary. Since
$M^{n}$ is finite-ended, there is no loss of generality in assuming that
$M^{n}$ is 1-ended. By taking a product with $\mathbb{S}^{k}$ ($k\geq2$) if
necessary, we may arrange that $n\geq6$, without changing the fundamental
group at infinity. So, by the semistability conclusion of Theorem
~\ref{semistable} combined with the Generalized $1$-neighborhood Theorem
\cite[Th.4]{Gu1}, we may choose a generalized $1$-neighborhood end structure
$\left\{  N_{i}\right\}  $ for which each bonding map in the inverse sequence
\begin{equation}
\pi_{1}\left(  N_{0},p_{0}\right)  \overset{\lambda_{1}}{\twoheadleftarrow}%
\pi_{1}\left(  N_{1},p_{1}\right)  \overset{\lambda_{2}}{\twoheadleftarrow}%
\pi_{1}\left(  N_{2},p_{2}\right)  \overset{\lambda_{3}}{\twoheadleftarrow
}\cdots
\end{equation}
is surjective. Abelianization gives an inverse sequence%
\begin{equation}
H_{1}\left(  N_{0}\right)  \overset{\lambda_{1\ast}}{\twoheadleftarrow}%
H_{1}\left(  N_{1}\right)  \overset{\lambda_{2\ast}}{\twoheadleftarrow}%
H_{1}\left(  N_{2}\right)  \overset{\lambda_{3\ast}}{\twoheadleftarrow}%
\cdots\text{.} \label{Sequence: H_1 sequence}%
\end{equation}
which, by Theorem ~\ref{semistable}, is stable. It follows that all but
finitely many of the epimorphisms in (\ref{Sequence: H_1 sequence}) are
isomorphisms, so by omitting finitely many terms (then relabeling), we may
assume all bonds in (\ref{Sequence: H_1 sequence}) are isomorphisms. A
term-by-term application of Lemma ~\ref{Lemma: group homology} gives the following.

\begin{proposition}
\label{Prop: near perfect 1}Every 1-ended inward tame manifold $M^{n}$ with
compact boundary admits a generalized $1$-neighborhood end structure $\left\{
N_{i}\right\}  $ for which all bonding maps in the sequence $\left\{  \pi
_{1}\left(  N_{i},p_{i}\right)  ,\lambda_{i}\right\}  $ are surjective and
each $\ker\lambda_{i}$ is $\pi_{1}\left(  N_{i},p_{i}\right)  $-perfect; in
other words, if $\left\{  L_{i}=\pi_{1}\left(  N_{i},p_{i}\right)  \right\}  $
is the maximal augmentation, then $\left\{  \pi_{1}\left(  N_{i},p_{i}\right)
,\lambda_{i}\right\}  $ satisfies the $\left\{  L_{i}\right\}  $-perfectness property.
\end{proposition}

Theorem \ref{Theorem: inward tame implies AP-semistability} is a stronger
version of Proposition \ref{Prop: near perfect 1}. For clarity, we restate it
in a similar form.

\begin{proposition}
\label{Prop: near perfect 2}Every 1-ended inward tame manifold $M^{n}$ with
compact boundary admits a generalized $1$-neighborhood end structure $\left\{
N_{i}\right\}  $ for which all bonding maps in the sequence $\left\{  \pi
_{1}\left(  N_{i},p_{i}\right)  ,\lambda_{i}\right\}  $ are surjective and, if
we let $K_{i}=\ker\lambda_{i}$ for each $i\geq1$ (the standard augmentation),
then $K_{i}$ is $\lambda_{i}^{-1}\left(  K_{i-1}\right)  $-perfect for all
$i\geq2$. In other words, $\left\{  \pi_{1}\left(  N_{i},p_{i}\right)
,\lambda_{i}\right\}  $ satisfies the $\left\{  K_{i}\right\}  $-perfectness
property; so $M^{n}$ has $\mathcal{AP}$-semistable fundamental group at infinity.

\begin{proof}
Assume the sequence $\left\{  N_{i}\right\}  $ was chosen so that, for each
$i$, $N_{i+1}$ is sufficiently small that a taming homotopy $H^{i}$ pulls
$N_{i}$ into $A_{i}=N_{i}-\operatorname{int}N_{i+1}$, i.e., $\overline
{H_{1}^{i}\left(  N_{i}\right)  }\subseteq A_{i}$, and $N_{i+3}$ is
sufficiently small that $H^{i}\left(  \partial N_{i+2}\times\left[
0,1\right]  \right)  \cap N_{i+3}=\varnothing$. By compactness of
$\overline{H_{1}^{i}\left(  N_{i}\right)  }$ and $H^{i}\left(  \partial
N_{i+2}\times\left[  0,1\right]  \right)  $ those choices can be made.

Now let $i\geq2$ be fixed and $q_{i-2}:\widetilde{N}_{i-2}\rightarrow N_{i-2}$
be the universal covering projection. Let $\widetilde{A}_{i-2}=q_{i-2}%
^{-1}\left(  A_{i-2}\right)  $ and for $j>i-2$, $\widehat{N}_{j}=q_{i-2}%
^{-1}\left(  N_{j}\right)  $ and $\widehat{A}_{j}=p_{i-2}^{q-1}\left(
A_{j}\right)  .$ Then $\widetilde{N}_{i-2}\supset\widehat{N}_{i-1}%
\supset\widehat{N}_{i}\supset\widehat{N}_{i+1}$; and $H^{i-2}$ lifts to a
proper homotopy $\widetilde{H}^{i-2}$ that pulls $\widetilde{N}_{i-2}$ into
$\widetilde{A}_{i-2}$ and for which $\widetilde{H}^{i}\left(  \partial
\widehat{N}_{i}\times\left[  0,1\right]  \right)  $ misses $\widehat{N}_{i+1}%
$.

We may associate $\lambda_{i}^{-1}\left(  K_{i-1}\right)  $ with $\pi
_{1}\left(  \widehat{N}_{i}\right)  $ and $K_{i}$ with $\ker\left(  \pi
_{1}\left(  \widehat{N}_{i}\right)  \rightarrow\pi_{1}\left(  \widehat
{N}_{i-1}\right)  \right)  $. Thus, an arbitrary element of $K_{i}$ may be
viewed as a loop $\alpha$ in $\partial\widehat{N}_{i}$ that bounds a disk $D$
in $\widehat{A}_{i-1}$. To prove the Proposition, it suffices to show that
$\alpha$ bounds an orientable surface in $\widehat{N}_{i}$. By $\pi_{1}%
$-surjectivity and the fact that the $N_{j}$'s are generalized $1$%
-neighborhoods, $\alpha$ may be homotoped within $\widehat{A}_{i}$ to a loop
$\alpha_{0}$ in $\partial\widehat{N}_{i+1}$. Let $E$ be the cylinder in
$\widehat{A}_{i}$ between $\alpha$ and $\alpha_{0}$ traced out by that
homotopy. Then the disk $D\cup E$ may be viewed as an element $[\beta]\in
H_{2}\left(  \widehat{A}_{i}\cup\widehat{A}_{i-1},\partial\widehat{N}%
_{i+1}\right)  $. Let
\[
\widehat{f}:\partial\widehat{N}_{i}\times\lbrack0,1]\cup_{\partial\widehat
{N}_{i}\times\{0\}}\widehat{A}_{i}\rightarrow\widetilde{A}_{i-2}\cup
\widehat{A}_{i-1}\cup\widehat{A}_{i}%
\]
be the identity on $\widehat{A}_{i}$ and $\left.  \widetilde{H}^{i-2}%
\right\vert $ on $\partial\widehat{N}_{i}\times\lbrack0,1]$. By PL
transversality theory (see \cite{RS1} or \cite[\S II.4]{BRS}), we may---after
a small proper adjustment that does not alter $\widehat{f}$ on $(\partial
\widehat{N}_{i}\times\{0,1\})\cup\widehat{A}_{i}$---assume that $\widehat
{f}^{-1}\left(  \widehat{A}_{i-1}\cup\widehat{A}_{i}\right)  $ is a manifold
with boundary that is a homeomorphism over a collar neighborhood of
$\partial\widehat{N}_{i+1}$. Let $\widehat{C}$ be the component of
$\widehat{f}^{-1}\left(  \widehat{A}_{i-1}\cup\widehat{A}_{i}\right)  $
containing that neighborhood. Then, by local characterization of degree,
$\left.  \widehat{f}\right\vert :\widehat{C}\rightarrow\widehat{A}_{i-1}%
\cup\widehat{A}_{i}$ is a proper degree $1$ map, and $\left.  \widehat
{f}\right\vert ^{-1}\left(  \partial\widehat{N}_{i+1}\right)  =\partial
\widehat{N}_{i+1}$. Thus we have a surjection%
\[
\left.  \widehat{f}\right\vert _{\ast}:H_{2}\left(  \widehat{C},\partial
\widehat{N}_{i+1}\right)  \rightarrow H_{2}\left(  \widehat{A}_{i}\cup
\widehat{A}_{i-1},\partial\widehat{N}_{i+1}\right)  \text{.}%
\]
Let $\left[  \beta^{\prime}\right]  $ be a preimage of $\left[  \beta\right]
$. We may assume that $\beta^{\prime}$ is an orientable surface with boundary
in $\widehat{C}$. Since $\widehat{f}$ is the identity on $\partial\widehat
{N}_{i+1}$, $\partial\beta^{\prime}$ is homologous in $\partial\widehat
{N}_{i+1}$ to $\partial\beta=\alpha_{0}$. So, without loss of generality, we
man assume that $\partial\beta^{\prime}=\alpha_{0}$. Since $\widehat{C}$ lies
in $\partial\widehat{N}_{i}\times\lbrack0,1]\cup_{\partial\widehat{N}%
_{i}\times\{0\}}\widehat{A}_{i}$, we may push $\beta^{\prime}$, rel boundary,
into $\widehat{A}_{i}$. This provides an orientable surface in $\widehat
{A}_{i}$ with boundary $\alpha_{0}$. Gluing the cylinder $E$ to that surface
along $\alpha_{0}$ produces the bounding surface for $\alpha$ that we desire.
\end{proof}
\end{proposition}

Early attempts to prove $\mathcal{P}$-semistability (hence
pseudo-collarability) with only an assumption of inward tameness, were brought
to a halt by the discovery of a key example presented in \cite{GT1}. Ideas
contained in that example play an important role here, so we provide a quick review.

An easy way to denote normal subgroups will be helpful. Let $G$ be a group and
$S\subseteq G$. The \emph{normal closure of S in G\/} is the smallest normal
subgroup of $G$ containing $S$. We denote it by $\operatorname*{ncl}(S,G)$.

\begin{example}
[Main Example from \cite{GT1}]\label{Example: 1} For all $n\geq6$, there exist
1-ended absolutely inward tame open $n$-manifolds with fundamental group
system%
\[
G_{0}\overset{\lambda_{1}}{\twoheadleftarrow}G_{1}\overset{\lambda_{2}%
}{\twoheadleftarrow}G_{2}\overset{\lambda_{3}}{\twoheadleftarrow}\cdots
\]
where
\[
G_{i}=\left\langle a_{0},a_{1},\cdots,a_{i}\mid a_{1}=[a_{1},a_{0}%
],a_{2}=\left[  a_{2},a_{1}\right]  ,\cdots,a_{i}=\left[  a_{i},a_{i-1}%
\right]  \right\rangle
\]
and $\lambda_{i}$ sends $a_{j}$ to $a_{j}$ for $0\leq j\leq i-1$ and $a_{i}$
to $1$.

By a largely algebraic argument, it was shown that these examples do not have
$\mathcal{P}$-semistable fundamental group at infinity, and thus, are not
pseudo-collarable. Notice, however, that each $K_{i}=\ker\lambda_{i}$ is the
normal closure of $a_{i}$ and $a_{i}=\left[  a_{i},a_{i-1}\right]  $ in
$G_{i}$; so $K_{i}\trianglelefteq\left[  K_{i},\lambda_{i}^{-1}\left(
K_{i-1}\right)  \right]  $. In other words, $\left\{  G_{i},\lambda
_{i}\right\}  $ satisfies the strong $\left\{  K_{i}\right\}  $-perfectness
property, and is therefore $\mathcal{SAP}$-semistable.

In addition to the above algebra, these examples have nice topological
properties. Although they do not contain small homotopy collar neighborhoods
of infinity, they do contain arbitrarily small generalized $1$-neighborhoods
of infinity $N$ for which $\partial N\hookrightarrow N$ is $%
\mathbb{Z}
$-homology equivalence. In fact, they contain a sequence $\left\{
N_{i}\right\}  $ of generalized $1$-neighborhoods of infinity with $\pi
_{1}\left(  N_{i}\right)  \cong G_{i}$ and $\partial N_{i}\hookrightarrow
N_{i}$ a $%
\mathbb{Z}
\lbrack G_{i-1}]$-homology equivalence.

These observations provide much of the motivation for the remainder of this paper.
\end{example}

\section{Generalizing one-sided h-cobordisms, homotopy collars and
pseudo-collars\label{Section: Generalizing 1-sided cobordisms, homotopy collars, and pseudocollars}%
}

We begin developing ideas for placing Example \ref{Example: 1} into a general
context. We will see that end structures like those found in that example are
possible only when kernels satisfy a strong relative perfectness condition.
Conversely, we will show that whenever such a group theoretic condition is
present, a corresponding `near pseudo-collar' structure is attainable.

We have already defined pseudo-collar structure on the end of a manifold
$M^{n}$ to be an end structure $\left\{  N_{i}\right\}  $ for which each
$N_{i}$ is a homotopy collar, i.e., each $\partial N_{i}\hookrightarrow N_{i}$
is a homotopy equivalence. The existence of such a structure allows us to
express each $N_{i}$ as a union%
\[
N_{i}=W_{i}\cup W_{i+1}\cup W_{i+2}\cup\cdots
\]
where $W_{i}=N_{i}-\operatorname*{int}N_{i+1}$, and each triple $\left(
W_{i},\partial N_{i},\partial N_{i+1}\right)  $ is a compact \emph{one-sided
h-cobordism} in the sense that $\partial N_{i}\hookrightarrow W_{i}$ is a
homotopy equivalence (and $\partial N_{i+1}\hookrightarrow W_{i}$ is probably
not). One-sided cobordisms play an important role in manifold topology in
general, and the topology of ends in particular. See \cite[\S 4]{Gu1} for
details. For later use, we review a few key properties of one-sided
h-cobordisms. See, for example, \cite[Theorem 2.5]{GT1}

\begin{theorem}
\label{one-sided}Let $(W,P,Q)$ be a compact cobordism between closed manifolds
with $P\hookrightarrow W$ a homotopy equivalence. Then

\begin{enumerate}
\item $P\hookrightarrow W$ and $Q\hookrightarrow W$ are $%
\mathbb{Z}
\left[  \pi_{1}\left(  W\right)  \right]  $-homology equivalences, i.e.,
\newline$H_{\ast}\left(  W,P;%
\mathbb{Z}
\left[  \pi_{1}\left(  W\right)  \right]  \right)  \allowbreak=\allowbreak
0\allowbreak=\allowbreak H_{\ast}\left(  W,Q;%
\mathbb{Z}
\left[  \pi_{1}\left(  W\right)  \right]  \right)  $

\item $\pi_{1}\left(  Q\right)  \rightarrow\pi_{1}\left(  W\right)  $ is
surjective, and $\allowbreak$

\item $K=\ker\left(  \pi_{1}\left(  Q\right)  \rightarrow\pi_{1}\left(
W\right)  \right)  $ is perfect \medskip
\end{enumerate}
\end{theorem}

Moving forward, we require generalizations of the fundamental concepts:
homotopy equivalence, homotopy collar, one-sided h-cobordism and
pseudo-collar. They are as follows:\medskip

\begin{itemize}
\item Let $\left(  X,A\right)  $ be a CW-pair for which $i:A\hookrightarrow X$
induces a $\pi_{1}$-isomorphism, and let $L\trianglelefteq\pi_{1}\left(
A\right)  $. Call $i$ a $\left(  \operatorname{mod}L\right)  $\emph{-homotopy
equivalence} if $H_{\ast}\left(  X,A;%
\mathbb{Z}
\mathbb{[\pi}_{1}\left(  A\right)  /L]\right)  =0$ for all $\ast$. Extension
to arbitrary maps is accomplished by use of mapping cylinders.

\item A manifold $N$ with compact boundary is a $\left(  \operatorname{mod}%
L\right)  $\emph{-homotopy collar} if $L\trianglelefteq\pi_{1}\left(  \partial
N\right)  $ and $\partial N\hookrightarrow N$ is a $\left(  \operatorname{mod}%
L\right)  $-homotopy equivalence.

\item Let $\left(  W,P,Q\right)  $ be a compact cobordism between closed
manifolds and $L\trianglelefteq\pi_{1}\left(  W\right)  $. We call $\left(
W,P,Q\right)  $ a $\left(  \operatorname{mod}L\right)  $\emph{-one-sided
h-cobordism} if $i:P\hookrightarrow W$ is a $\left(  \operatorname{mod}%
L\right)  $-homotopy equivalence and $j:Q\hookrightarrow W$ induces a
surjection on fundamental groups.

\item Let $\left\{  N_{i}\right\}  $ be a generalized $1$-neighborhood end
structure on a manifold $M^{n}$, chosen so that the bonding maps in
\[
\pi_{1}\left(  N_{0}\right)  \overset{\lambda_{1}}{\twoheadleftarrow}\pi
_{1}\left(  N_{1}\right)  \overset{\lambda_{2}}{\twoheadleftarrow}\pi
_{1}\left(  N_{2}\right)  \overset{\lambda_{3}}{\twoheadleftarrow}%
\cdots\text{.}%
\]
are surjective, and let $\left\{  L_{i}\right\}  $ be an augmentation of this
sequence. Call $\left\{  N_{i}\right\}  $ a $\operatorname{mod}\left(
\left\{  L_{i}\right\}  \right)  $ \emph{pseudo-collar structure }if each
$\partial N_{i}\hookrightarrow N_{i}$ is a $\left(  \operatorname{mod}%
L_{i}\right)  $-homotopy equivalence.\medskip
\end{itemize}

\begin{remark}
\emph{i)}\textbf{ }\emph{Each of the above definitions reduces to its
traditional counterpart when the subgroup(s) involved are trivial.}

\emph{ii)} \emph{In the generalization of one-sided h-cobordism, we} require
$j_{\#}:\pi_{1}(Q)\rightarrow\pi_{1}(W)$\emph{ to be surjective---a condition
that is automatic when }$L=\left\{  1\right\}  $\emph{, but not in general.
Analogs of the other two assertions of Theorem \ref{one-sided} will be shown
to follow.}

\emph{iii)}\textbf{ }\emph{For the maximal augmentation, the generalization of
pseudo-collar requires only that} \emph{each }$\partial N_{i}\hookrightarrow
N_{i}$\emph{ be a }$%
\mathbb{Z}
$\emph{-homology equivalence; whereas, for the trivial augmentation, we have a
genuine pseudo-collar. The key dividing line between those extremes occurs
when }$\left\{  L_{i}\right\}  $\emph{ is a small augmentation (}$L_{i}%
\leq\ker\lambda_{i}$\emph{ for all }$i$\emph{). In those cases, we call
}$\left\{  N_{i}\right\}  $\emph{ a }near pseudo-collar\emph{ structure, and
say that a 1-ended }$M^{n}$\emph{ with compact boundary is }nearly
pseudo-collarable\emph{ if it admits such a structure. The geometric
significance of the `small augmentation' requirement will become clear in the
proof of Theorem \ref{NPCT technical version}. Further discussion of that
topic is contained in \S \ref{Section: Remaining Questions}.\medskip}
\end{remark}

The following lemma adds topological meaning to the definition of $\left(
\operatorname{mod}L\right)  $-homotopy equivalence.

\begin{lemma}
\label{Lemma: (modL) h.e.}Let $\left(  X,A\right)  $ be a CW-pair for which
$i:A\hookrightarrow X$ induces a $\pi_{1}$-isomorphism, $L\trianglelefteq
\pi_{1}\left(  A\right)  $, and $S\subseteq L$ for which $\operatorname*{ncl}%
(S,\pi_{1}\left(  A\right)  )=L$. Obtain $A^{\prime}$ from $A$ by attaching a
$2$-disk $D_{s}$ along each $s\in S$; let $X^{\prime}=X\cup(\bigcup_{s\in
S}D_{s})$, and $i^{\prime}:A^{\prime}\hookrightarrow X^{\prime}$. Then $i$ is
a $\left(  \operatorname{mod}L\right)  $-homotopy equivalence if and only if
$i^{\prime}$ is a homotopy equivalence.
\end{lemma}

\begin{proof}
Let $p:\widehat{X}\rightarrow X$ be the covering projection corresponding to
$L$. Then $\widehat{A}=p^{-1}\left(  A\right)  $ is the cover of $A$
corresponding to $L$. Viewing $S$ as a collection of loops in $A$ and
$\widehat{S}$ the set of all lifts of those loops, then attaching $2$-disks to
$\widehat{A}$ (and, simultaneously $\widehat{X}$) along $\widehat{S}$ produces
universal covers $\widetilde{A}^{\prime}$ and $\widetilde{X}^{\prime}$.

Assume now that $i:A\hookrightarrow X$ is a $(\operatorname{mod}L)$-homotopy
equivalence. Then by Shapiro's Lemma \cite[p.100]{DK}, $H_{\ast}\left(
\widehat{X},\widehat{A};%
\mathbb{Z}
\right)  =0$, so by excision $H_{\ast}\left(  \widetilde{X}^{\prime
},\widetilde{A}^{\prime};%
\mathbb{Z}
\right)  =0$. Since both spaces are simply connected, the relative Hurewicz
Theorem implies that $\pi_{\ast}\left(  \widetilde{X}^{\prime},\widetilde
{A}^{\prime}\right)  =0$; therefore $\pi_{\ast}\left(  X^{\prime},A^{\prime
}\right)  =0$. By Whitehead's Theorem $i^{\prime}$ is a homotopy equivalence.

Conversely, if $i^{\prime}$ is a homotopy equivalence, then its lift
$\widetilde{A}^{\prime}\hookrightarrow\widetilde{X}^{\prime}$ is a homotopy
equivalence. Therefore $H_{\ast}\left(  \widetilde{X}^{\prime},\widetilde
{A}^{\prime};%
\mathbb{Z}
\right)  =0$, so by excision $H_{\ast}\left(  \widehat{X},\widehat{A};%
\mathbb{Z}
\right)  =0$, and by Shapiro's Lemma $H_{\ast}\left(  X,A;%
\mathbb{Z}
\mathbb{[\pi}_{1}\left(  A\right)  /L]\right)  =0$.
\end{proof}

The following is a useful corollary.

\begin{lemma}
\label{Lemma: (modQ) h.e.} Let $\left(  X,A\right)  $ be a CW-pair for which
$i:A\hookrightarrow X$ induces a $\pi_{1}$-isomorphism. Suppose
$L\trianglelefteq\pi_{1}\left(  A\right)  $. If $H_{\ast}\left(  X,A;%
\mathbb{Z}
\lbrack\pi_{1}\left(  A\right)  /L]\right)  =0$, then $H_{\ast}\left(  X,A;%
\mathbb{Z}
\lbrack\pi_{1}\left(  A\right)  /J]\right)  =0$ for any $J$ with $L<J\unlhd
\pi_{1}\left(  A\right)  $. In particular, $H_{\ast}\left(  X,A;%
\mathbb{Z}
\right)  =0\medskip$
\end{lemma}

The next observation is a direct analog of Theorem \ref{one-sided}.

\begin{theorem}
\label{Thm: one-sided general} Let $\left(  W,P,Q\right)  $ be a compact
$(\operatorname{mod}L)$-one-sided h-cobordism between closed manifolds with
$L\trianglelefteq\pi_{1}\left(  W\right)  $. Let $j:Q\hookrightarrow W$ and
$L^{\prime}=j_{\#}^{-1}\left(  L\right)  $. Then

\begin{enumerate}
\item both $P\hookrightarrow W$ and $Q\hookrightarrow W$ are $%
\mathbb{Z}
\mathbb{[\pi}_{1}\left(  W\right)  /L]$-homology equivalences, i.e.,
\newline$H_{\ast}\left(  W,P;%
\mathbb{Z}
\mathbb{[\pi}_{1}\left(  W\right)  /L]\right)  =\allowbreak0=\allowbreak
H_{\ast}\left(  W,Q;%
\mathbb{Z}
\mathbb{[\pi}_{1}\left(  W\right)  /L]\right)  $, and

\item $K=\ker j_{\#}$ $\trianglelefteq\pi_{1}\left(  Q\right)  $ is strongly
$L^{\prime}$-perfect.
\end{enumerate}
\end{theorem}

\begin{proof}
First note that by the surjectivity of $j_{\#}:\pi_{1}(Q)\rightarrow\pi
_{1}(W)$, there is a canonical isomorphism $\mathbb{\pi}_{1}\left(  Q\right)
/L^{\prime}\overset{\cong}{\rightarrow}\mathbb{\pi}_{1}\left(  W\right)  /L$
that is assumed throughout. Let $p:\widehat{W}_{L}\rightarrow W$ be the
covering projection corresponding to $L$, $\widehat{P}=p^{-1}\left(  P\right)
$ and $\widehat{Q}=p^{-1}\left(  Q\right)  $. Then both $\widehat{P}$ and
$\widehat{Q}$ are connected, and their projections onto $P$ and $Q$ are the
coverings corresponding to $L$ and $L^{\prime}$.

The assertion that $H_{\ast}\left(  W,P;%
\mathbb{Z}
\mathbb{[\pi}_{1}\left(  W\right)  /L]\right)  =0$ is part of the hypothesis,
and (by Shapiro's Lemma \cite[p.100]{DK}) equivalent to the assumption that
$H_{\ast}\left(  \widehat{W}_{L},\widehat{P};%
\mathbb{Z}
\right)  =0$. To show that $H_{\ast}\left(  W,Q;%
\mathbb{Z}
\mathbb{[\pi}_{1}\left(  W\right)  /L]\right)  $ vanishes in all dimensions,
it suffices to show that $H_{\ast}\left(  \widehat{W}_{L},\widehat{Q};%
\mathbb{Z}
\right)  =0$. This will follow from Poincar\'{e} duality for noncompact
manifolds if we can verify:\medskip

\noindent\emph{Claim. }$H_{f}^{\ast}\left(  \widehat{W}_{L},\widehat{P};%
\mathbb{Z}
\right)  =0$\emph{, where the `}$f$\emph{' indicates cellular cohomology based
on finite cochains. (See \cite[Ch. 12]{Ge}.)}\medskip

Applying Lemma \ref{Lemma: (modL) h.e.}, attach $2$-cells to $W$ along a
collection $S$ of loops in $P$ to kill $L$, obtaining spaces $P^{\prime}$ and
$W^{\prime}$, and a homotopy equivalence $P^{\prime}\hookrightarrow W^{\prime
}$. Since $W$ is compact, any strong deformation retraction of $W^{\prime}$
onto $P^{\prime}$ is proper, and hence, lifts to a proper strong deformation
retraction of universal covers $\widetilde{W}^{\prime}$ onto $\widetilde
{P}^{\prime}$ \cite[\S 10.1]{Ge}. It follows that $H_{f}^{\ast}\left(
\widehat{W}^{\prime},\partial\widehat{N}_{i-1}^{\prime};%
\mathbb{Z}
\right)  =0$. Both universal covers are obtained by attaching disks along the
collection $\widehat{S}$ of lifts to $\widehat{P}$ and $\widehat{W}$ of the
loops in $S$. By excising the interiors of those disks, we conclude that
$H_{f}^{\ast}\left(  \widehat{W},\partial\widehat{N};%
\mathbb{Z}
\right)  =0$.\medskip

To verify assertion (2), consider the short exact sequence
\[
1\rightarrow K\rightarrow L^{\prime}\overset{q}{\longrightarrow}L^{\prime
}/K\rightarrow1
\]
where $L^{\prime}/K$ may be identified with $L$. Lemma
\ref{Lemma: 5 term exact} provides the 5-term exact sequence
\[
H_{2}\left(  L^{\prime};%
\mathbb{Z}
\right)  \overset{q_{\ast2}}{\longrightarrow}H_{2}\left(  L^{\prime}/K;%
\mathbb{Z}
\right)  \rightarrow K/\left[  K,L^{\prime}\right]  \rightarrow H_{1}\left(
L^{\prime};%
\mathbb{Z}
\right)  \overset{q_{\ast1}}{\longrightarrow}H_{1}\left(  L^{\prime}/K;%
\mathbb{Z}
\right)  \rightarrow0\text{.}%
\]
from which the $L^{\prime}$-perfectness of $K$ can be deduced by showing that
$q_{\ast2}$ is an epimorphism and $q_{\ast1}$ an isomorphism.

Since $\widehat{Q}\hookrightarrow\widehat{W}_{L}$ induces $q:L^{\prime
}\rightarrow L$ and since $H_{2}\left(  \widehat{W}_{L},\widehat{Q};%
\mathbb{Z}
\right)  =0$, the the long exact sequence for that pair ensures that
$H_{1}\left(  L^{\prime};%
\mathbb{Z}
\right)  \overset{\cong}{\rightarrow}H_{1}\left(  L;%
\mathbb{Z}
\right)  $. In addition, the surjectivity of $H_{2}\left(  \widehat{Q};%
\mathbb{Z}
\right)  \rightarrow H_{2}\left(  \widehat{W}_{L};%
\mathbb{Z}
\right)  $ combines with Lemma \ref{Lemma: induced surjections on H_2} to
imply surjectivity of $H_{2}\left(  L^{\prime};%
\mathbb{Z}
\right)  \rightarrow H_{2}\left(  L;%
\mathbb{Z}
\right)  $. \medskip
\end{proof}

\section{The structure of inward tame ends}

With all necessary definitions in place, we are ready to prove the second main
theorem described in the introduction. We begin by stating a strong form of
the theorem, written in the style of earlier characterization theorems from
\cite{Si} and \cite{GT2}.

\begin{theorem}
[Near Pseudo-collarability Characterization]\label{NPCT technical version}A
1-ended $n$-manifold $M^{n}$ ($n\geq6$) with compact boundary is nearly
pseudo-collarable iff each of the following conditions holds:

\begin{enumerate}
\item $M^{n}$ is inward tame,

\item the fundamental group at infinity is $\mathcal{SAP}$-semistable, and

\item $\sigma_{\infty}\left(  M^{n}\right)  =0\in\widetilde{K}_{0}\left(
\pi_{1}\left(  \varepsilon(M^{n}\right)  \right)  )$.
\end{enumerate}
\end{theorem}


Recall that condition (2) presumes the existence of a representation of
$\pi_{1}\left(  \varepsilon(M^{n}\right)  $ of the form%
\begin{equation}
G_{0}\overset{\lambda_{1}}{\twoheadleftarrow}G_{1}\overset{\lambda_{2}%
}{\twoheadleftarrow}G_{2}\overset{\lambda_{3}}{\twoheadleftarrow}%
\cdots\label{sequence: pi1-setup for main theorem}%
\end{equation}
with a small augmentation $\left\{  L_{i}\right\}  $ ($L_{i}\trianglelefteq
K_{i}=\ker\lambda_{i}$, for all $i$) so that each $K_{i}$ is strongly $J_{i}%
$-perfect, where $J_{i}=\lambda_{i}^{-1}\left(  L_{i-1}\right)  $.

\begin{proof}
First we verify that a nearly pseudo-collarable 1-ended manifold with compact
boundary must satisfy conditions (1)-(3).

The hypothesis provides a generalized $1$-neighborhood end structure $\left\{
N_{i}\right\}  $ on $M^{n}$ with group data
\begin{equation}
G_{0}\overset{\lambda_{1}}{\twoheadleftarrow}G_{1}\overset{\lambda_{2}%
}{\twoheadleftarrow}G_{2}\overset{\lambda_{3}}{\twoheadleftarrow}\cdots.
\label{inverse sequence in Structure Theorem}%
\end{equation}
($G_{i}=\pi_{1}\left(  N_{i}\right)  $) and a small augmentation $\left\{
L_{i}\right\}  $ ($L_{i}\trianglelefteq K_{i}=\ker\lambda_{i}$) such that each
$N_{i}$ is a $\operatorname*{mod}\left(  L_{i}\right)  $-homotopy collar.

To simultaneously verify (1) and (3), it suffices to exhibit a cofinal
sequence of clean neighborhoods of infinity, each having finite homotopy type.
Lemma \ref{Lemma: (modQ) h.e.} ensures that each $N_{i}$ is a
$\operatorname{mod}\left(  K_{i}\right)  $-homotopy collar, and since each
$\lambda_{i}$ is a surjection between finitely presented groups, each $K_{i}$
is finitely generated as a normal subgroup of $G_{i}$. Let $i$ be fixed, and
$A=\left\{  \alpha_{j}\right\}  $ a finite collection of loops in $\partial
N_{i}$ that normally generates $K_{i}$ in $G_{i}$. By Lemma
\ref{Lemma: (modL) h.e.}, if we abstractly attach a $2$-disk $\Delta_{j}^{2}$
along each $\alpha_{j}$, we obtain a homotopy equivalence
\[
\partial N_{i}\bigcup\left(  \cup\Delta_{j}^{2}\right)  \hookrightarrow
N_{i}\bigcup\left(  \cup\Delta_{j}^{2}\right)  .
\]
In particular, $N_{i}\bigcup\left(  \cup\Delta_{j}^{2}\right)  $ has the
homotopy type of a finite complex. But, since each $\alpha_{j}$ represents an
element of $\ker\lambda_{i}$, we may assume that each $\Delta_{j}^{2}$ is
properly embedded in $N_{i-1}-\operatorname*{int}\left(  N_{i}\right)  $. By
thickening these $2$-disks to $2$-handles, we obtain a clean neighborhood of
infinity $N_{i}^{\ast}$ with finite homotopy type, lying in $N_{i-1}$.

This leaves only $\mathcal{SAP}$-semistability to be checked. We will show
that (\ref{inverse sequence in Structure Theorem}) satisfies the strong
$\left\{  L_{i}\right\}  $-perfectness property; in other words, each $K_{i}$
is strongly $J_{i}$-perfect, where $J_{i}=\lambda_{i}^{-1}\left(
K_{i-1}\right)  $.

For each $i>0$, let $W_{i-1}=N_{i-1}-\operatorname*{int}\left(  N_{i}\right)
$.\medskip

\noindent\emph{Claim. }$\left(  W_{i-1},\partial N_{i-1},\partial
N_{i}\right)  $\emph{ is a }$(\operatorname{mod}L_{i-1})$\emph{-one-sided
h-cobordism.\medskip}

Fix $i$ and let $p:\widehat{N}_{i-1}\rightarrow N_{i-1}$ be the covering
corresponding to $L_{i-1}\trianglelefteq G_{i-1}=\pi_{1}\left(  N_{i-1}%
\right)  \cong\pi_{1}\left(  W_{i-1}\right)  $; let $\widehat{W}_{i-1}$ denote
$p^{-1}\left(  W_{i-1}\right)  $; and let $\widehat{N}_{i}$ denote
$p^{-1}\left(  N_{i}\right)  $. Then $\widehat{W}_{i-1}$ is the cover of
$W_{i-1}$ corresponding to $J_{i-1}$, and $\widehat{N}_{i}$ is the cover of
$N_{i}$ corresponding to $J_{i}\trianglelefteq G_{i}=\pi_{1}\left(
N_{i}\right)  $. By Lemma \ref{Lemma: (modQ) h.e.} and Shapiro's Lemma%
\[
0=H_{\ast}\left(  N_{i},\partial N_{i};%
\mathbb{Z}
\lbrack G_{i}/J_{i}]\right)  \cong H_{\ast}\left(  \widehat{N}_{i}%
,\partial\widehat{N}_{i}\partial N_{i};%
\mathbb{Z}
\right)  \text{,}%
\]
and from the long exact homology sequence for the triple $\left(  \widehat
{N}_{i-1},\widehat{W}_{i-1},\partial\widehat{N}_{i-1}\right)  $, excision, and
Shapiro's Lemma
\[
H_{\ast}\left(  \widehat{W}_{i-1},\partial\widehat{N}_{i-1};%
\mathbb{Z}
\right)  \cong H_{\ast}\left(  W_{i-1},\partial N_{i-1};%
\mathbb{Z}
\lbrack G_{i-1}/L_{i-1}]\right)  =0.
\]
The claim follows.\medskip

Finally, since the bonding map $G_{i-1}\overset{\lambda_{i}}{\twoheadleftarrow
}G_{i}$ is represented by the inclusion $W_{i-1}\hookleftarrow\partial N_{i}$,
$K_{i}$ is strongly $J_{i}$-perfect by Theorem \ref{Thm: one-sided general}.


For the converse, we must show that conditions (1)-(3) imply the existence of
a near pseudo-collar structure on $M^{n}$. Though the proof is rather
complicated, it follows the same outline as \cite{Gu1}, which followed the
original proof in \cite{Si}. For a full understanding, the reader should be
familiar with \cite{Gu1}. The new argument presented here generalizes the the
final portions of that proof. A concise review of \cite{Gu1} can be found in
\cite[\S 4]{GT2}.

In \cite{Gu1} and \cite{GT2} the goal was to improve arbitrarily small
neighborhoods of infinity to homotopy collars. That is impossible with our
weaker hypotheses; instead, the goal is to improve neighborhoods of infinity
to homotopy collars modulo certain subgroups of their fundamental groups.

By condition (2) the pro-isomorphism class of $\pi_{1}\left(  \varepsilon
\left(  M^{n}\right)  \right)  $ may be represented by a sequence%
\begin{equation}
G_{0}\overset{\lambda_{1}}{\twoheadleftarrow}G_{1}\overset{\lambda_{2}%
}{\twoheadleftarrow}G_{2}\overset{\lambda_{3}}{\twoheadleftarrow}\cdots
\end{equation}
of finitely presented groups, along with a small augmentation $\left\{
L_{i}\right\}  $ ($L_{i}\trianglelefteq K_{i}=\ker\lambda_{i}$, for all $i$)
so that each $K_{i}$ is strongly $J_{i}$-perfect, where $J_{i}=\lambda
_{i}^{-1}\left(  L_{i-1}\right)  $.

By \cite[Lemma 8]{Gu1} there is a sequence $\left\{  N_{i}\right\}  $ of
generalized $1$-neighborhoods of infinity whose inverse sequence of
fundamental groups is isomorphic to a subsequence of $\left\{  G_{i}\right\}
$.
\[%
\begin{array}
[c]{ccccccccc}%
G_{i_{0}} & \overset{\lambda_{i_{0}+1,i_{1}}}{\twoheadleftarrow} & G_{i_{1}} &
\overset{\lambda_{i_{1}+1,i_{2}}}{\twoheadleftarrow} & G_{i_{2}} &
\overset{\lambda_{i_{2}+1,i_{3}}}{\twoheadleftarrow} & G_{i_{3}} &
\overset{\lambda_{i_{3}+1,i_{4}}}{\twoheadleftarrow} & \cdots\\
\updownarrow\cong &  & \updownarrow\cong &  & \updownarrow\cong &  &
\updownarrow\cong &  & \\
\pi_{1}\left(  N_{0},p_{0}\right)  & \overset{inc_{\#}}{\twoheadleftarrow} &
\pi_{1}\left(  N_{1},p_{1}\right)  & \overset{inc_{\#}}{\twoheadleftarrow} &
\pi_{1}\left(  N_{2},p_{2}\right)  & \overset{inc_{\#}}{\twoheadleftarrow} &
\pi_{1}\left(  N_{3},p_{3}\right)  & \overset{inc_{\#}}{\twoheadleftarrow} &
\cdots
\end{array}
\]
This diagram and Proposition \ref{Prop: nearly sub} ensure that, for each $j$,
$\ker\left(  \lambda_{i_{j-1}+1,i_{j}}\right)  $ is strongly\newline%
$\lambda_{i_{j-1}+1,i_{j}}^{-1}\left(  L_{i_{j-1}}\right)  $-perfect. So by
passing to this subsequence and relabeling, we may assume that sequence
(\ref{sequence: pi1-setup for main theorem}) and the corresponding subgroup
data matches the fundamental group data of $\left\{  N_{i}\right\}  $. Note
here that the $J$-groups (which are not viewed as part of the original data)
are not the same as the previous $J$-groups; they are now preimages of
\emph{compositions} of the original bonding maps.

Next we inductively improve the sequence $\{N_{j}\}$ to generalized
$k$-neighborhoods of infinity for increasing values of $k$, up to $k=n-3$. We
must frequently pass to subsequences, however, each improvement of a given
$N_{j}$ leaves its fundamental group and that of $\partial N_{i}$ intact; so
at each stage, the `new' fundamental group data will be a subsequence of the
original (\ref{sequence: pi1-setup for main theorem}), along with the
subsequence augmentation. The $J$-groups will change as per their definition,
but, by Proposition \ref{Prop: nearly sub}, we always maintain the appropriate
strong relative perfectness condition.

This neighborhood improvement process uses only the hypothesis that $M^{n}$ is
inward tame; it is identical that used in \cite[Th. 5]{Gu1} and outlined in
\cite[Theorem 3.2]{GT2}.
To save on notation we relabel the neighborhood sequences and their
corresponding groups at each stage, designating the resulting cofinal sequence
of generalized $\left(  n-3\right)  $-neighborhoods of infinity by $\left\{
N_{i}\right\}  $, with $G_{i}=\pi_{1}\left(  N_{i}\right)  $, $\lambda
_{i}:G_{i}\rightarrow G_{i-1}$ the corresponding homomorphism, $L_{i}%
\trianglelefteq K_{i}=\ker\lambda_{i}$, and $J_{i}=\lambda_{i}^{-1}\left(
L_{i-1}\right)  $.

For each $i$, let $R_{i}=N_{i}-\overset{\circ}{N}_{i+1}$ and consider the
collection of cobordisms $\{\left(  R_{i},\partial N_{i},\partial
N_{i+1}\right)  \}$. The following summary comprises the contents of Lemmas 11
and 12 of \cite{Gu1}, along with new hypotheses regarding kernels.

\begin{enumerate}
\item[i)] Each $N_{i}$ is a generalized $(n-3)$-neighborhood of infinity.

\item[ii)] Each induced bonding map $\pi_{1}\left(  N_{i}\right)
\twoheadleftarrow\pi_{1}\left(  N_{i+1}\right)  $ is surjective.

\item[iii)] Each inclusion $\partial N_{i}\hookrightarrow R_{i}\hookrightarrow
N_{i}$ induces a $\pi_{1}$-isomorphism.

\item[iv)] Each $\partial N_{i+1}\hookrightarrow R_{i}$ induces a $\pi_{1}%
$-epimorphism with kernel strongly $J_{i}$-perfect.

\item[v)] $\pi_{k}(R_{i},\partial N_{i})=0$ for all $k<n-3$ and all $i$.

\item[vi)] Each $\left(  R_{i},\partial N_{i},\partial N_{i+1}\right)  $
admits a handle decomposition based on $\partial N_{i}$ containing handles
only of index $\left(  n-3\right)  $ and $\left(  n-2\right)  $.

\item[vii)] Each $N_{i}$ admits an infinite handle decomposition with handles
only of index $\left(  n-3\right)  $ and $\left(  n-2\right)  $.

\item[viii)] Each $\left(  N_{i},\partial N_{i}\right)  $ has the homotopy
type of a relative CW pair $\left(  K_{i},\partial N_{i}\right)  $
with\newline$\dim\left(  K_{i}-\partial N_{i}\right)  \allowbreak
\leq\allowbreak n-2$.
\end{enumerate}


The obvious next goal is attempting to improve the $N_{i}$ to generalized
$\left(  n-2\right)  $-neigh\-bor\-hoods of infinity, which by item viii)
would necessarily be homotopy collars. In previous work \cite{Si}, \cite{Gu1}
and \cite{GT2}, that is the final (also the most difficult and interesting)
step. The same is true here, where the weakened hypotheses create greater
difficulties and the strategy and end goal must eventually be altered. For
now, we continue with the earlier strategies by turning attention to
$\pi_{n-2}\left(  N_{i},\partial N_{i}\right)  \cong H_{n-2}(\tilde{N}%
_{i},\partial\widetilde{N}_{i})$, which may be viewed as a $\mathbb{Z[\pi}%
_{1}N_{i}]$-module $H_{n-2}(N_{i},\partial N_{i};%
\mathbb{Z}
\mathbb{[\pi}_{1}N_{i}])$. The content of \cite[Lemma 13]{Gu1} is given by the
next two items.

\begin{enumerate}
\item[ix)] $H_{n-2}(\widetilde{N}_{i},\partial\widetilde{N}_{i})$ is a
finitely generated projective $\mathbb{Z[\pi}_{1}N_{i}]$-module.

\item[x)] As an element of $\widetilde{K}_{0}\left(  \mathbb{Z[\pi}_{1}%
N_{i}]\right)  $, $\left[  H_{n-2}(\widetilde{N}_{i},\partial\widetilde{N}%
_{i})\right]  =\left(  -1\right)  ^{n}\sigma\left(  N_{i}\right)  $, where
$\sigma\left(  N_{i}\right)  $ is the Wall finiteness obstruction for $N_{i}$.
\end{enumerate}

Taken together, these elements of $\widetilde{K}_{0}\left(  \mathbb{Z[\pi}%
_{1}N_{i}]\right)  $ determine the obstruction $\sigma_{\infty}\left(  \left(
\varepsilon(M^{n}\right)  \right)  $ found in condition (3). From now on we
assume that $\sigma_{\infty}\left(  M^{n}\right)  $ vanishes. This is
equivalent to assuming that each $\sigma\left(  N_{i}\right)  $ is the trivial
element of $\widetilde{K}_{0}\left(  \mathbb{Z[\pi}_{1}N_{i}]\right)  $, in
other words, each $H_{n-2}(\widetilde{N}_{i},\partial\widetilde{N}_{i})$ is a
stably free $\mathbb{Z[\pi}_{1}N_{i}]$-module. Therefore we have:

\begin{enumerate}
\item[xi)] By carving out finitely many trivial $\left(  n-3\right)  $-handles
from each $N_{i}$, we can arrange that $H_{n-2}(\widetilde{N}_{i}%
,\partial\widetilde{N}_{i})$ is a finitely generated free $\mathbb{Z[\pi}%
_{1}N_{i}]$-module.
\end{enumerate}

Item (xi) can be done so that these sets remain a generalized $\left(
n-3\right)  $-neighborhood of infinity, and so that their fundamental groups
and those of their boundaries are unchanged. Again, to save on notation, we
denote the improved collection by $\left\{  N_{i}\right\}  $. See \cite[Lemma
14]{Gu1} for details.

The finite generation of $H_{n-2}(\widetilde{N}_{i},\partial\widetilde{N}%
_{i})$ allows us to, after again passing to a subsequence and relabeling,
assume that

\begin{enumerate}
\item[xii)] $H_{n-2}(\widetilde{R}_{i},\partial\widetilde{N}_{i}%
)\twoheadrightarrow H_{n-2}(\widetilde{N}_{i},\partial\widetilde{N}_{i})$ is
surjective for each $i$.
\end{enumerate}

The long exact sequence for the triple $\left(  \widetilde{N}_{i}%
,\widetilde{R}_{i},\partial\widetilde{N}_{i}\right)  $ from there shows that

\begin{enumerate}
\item[xiii)] $H_{n-2}(\widetilde{R}_{i},\partial\widetilde{N}_{i}%
)\overset{\cong}{\rightarrow}H_{n-2}(\widetilde{N}_{i},\partial\widetilde
{N}_{i})$ is an isomorphism for each $i$ (hence, $H_{n-2}(\widetilde{R}%
_{i},\partial\widetilde{N}_{i})$ is a finitely generated free $\mathbb{Z[\pi
}_{1}R_{i}]$-module).
\end{enumerate}

As above, we may choose handle decompositions for the $R_{i}$ based on
$\partial N_{i}$ having handles only of index $n-3$ and $n-2$.

From now on, let $i$ be fixed. After introducing some trivial $\left(
n-3,n-2\right)  $-handle pairs, an algebraic lemma and some handle slides
allows us to obtain a handle decomposition of $R_{i}$ based on $\partial
N_{i}$ with $\left(  n-2\right)  $-handles $h_{1}^{n-2},h_{2}^{n-2}%
,\allowbreak\cdots,h_{r}^{n-2}$ and an integer $s\leq r$, such that the
subcollection $\{h_{1}^{n-2},h_{2}^{n-2},\allowbreak\cdots,h_{s}^{n-2}\}$ is a
free $\mathbb{Z}\left[  \pi_{1}R_{i}\right]  $-basis for $H_{n-2}\left(
\widetilde{R}_{i},\partial\widetilde{N}_{i}\right)  $. So we have:

\begin{enumerate}
\item[xiv)] the\emph{ }$\mathbb{Z}\left[  \pi_{1}R_{i}\right]  $-cellular
chain complex for $\left(  R_{i},\partial N_{i}\right)  $ may be expressed as%
\begin{equation}
0\rightarrow\left\langle h_{1}^{n-2},\cdots,h_{s}^{n-2}\right\rangle
\oplus\left\langle h_{s+1}^{n-2},\cdots,h_{r}^{n-2}\right\rangle
\overset{\partial}{\longrightarrow}\left\langle h_{1}^{n-3},\cdots,h_{t}%
^{n-3}\right\rangle \rightarrow0 \label{Chain complex for R_i}%
\end{equation}

\end{enumerate}

where

\begin{itemize}
\item $\left\langle h_{1}^{n-2},\cdots,h_{s}^{n-2}\right\rangle $ and
$\left\langle h_{s+1}^{n-2},\cdots,h_{r}^{n-2}\right\rangle $ represent free
$\mathbb{Z}\left[  \pi_{1}R_{i}\right]  $-submodules of $\widetilde{C}_{n-2}$
generated by the corresponding handles;\smallskip

\item $\left\langle h_{1}^{n-3},\cdots,h_{t}^{n-3}\right\rangle =\widetilde
{C}_{n-3}$\ is the free $\mathbb{Z}\left[  \pi_{1}R_{i}\right]  $-module
generated by the $\left(  n-3\right)  $-handles in $R_{i}$;\smallskip

\item $H_{n-2}(\widetilde{R}_{i},\partial\widetilde{N}_{i})=\ker
\partial=\left\langle h_{1}^{n-2},\cdots,h_{s}^{n-2}\right\rangle
\oplus\left\{  0\right\}  $; and\smallskip

\item $\partial$ takes $\left\{  0\right\}  \oplus\left\langle h_{s+1}%
^{n-2},\cdots,h_{r}^{n-2}\right\rangle $ injectively into $\left\langle
h_{1}^{n-3},\cdots,h_{t}^{n-3}\right\rangle $.\smallskip
\end{itemize}

\noindent Item xiv) and the preceding paragraph are the content \cite[Lemma
15]{Gu1}.\smallskip

To this point, we have only used the hypotheses of inward tameness and
triviality of the Wall obstruction to build the structure described by items
(i)-(xiv). All arguments used thus far appear in \cite{Gu1} and \cite{GT2},
with simpler analogs in \cite{Si}.

Under the $\pi_{1}$-stability hypothesis of \cite{Si}, $H_{n-2}(\widetilde
{R}_{i},\partial\widetilde{N}_{i})$ can now be killed by sliding the offending
$\left(  n-2\right)  $-handles $\left\{  h_{1}^{n-2},\cdots,h_{s}%
^{n-2}\right\}  $ off the $\left(  n-3\right)  $-handles and carving out their
interiors. Under the weaker $\mathcal{P}$-semistability hypothesis of
\cite{GT2}, a similar strategy works, but only after a significant preparatory
step, made possible by perfect kernels. In \cite{Gu1} an alternate strategy
was employed. Instead of killing $H_{n-2}(\widetilde{R}_{i},\partial
\widetilde{N}_{i})=\ker\partial$ by removing its generating handles $\left\{
h_{1}^{n-2},\cdots,h_{s}^{n-2}\right\}  $, the task was accomplished by
introducing new $\left(  n-3\right)  $-handles, which became images of the
$\left\{  h_{1}^{n-2},\cdots,h_{s}^{n-2}\right\}  $ under the resulting
boundary map, thereby trivializing the kernel. Complete discussions of these
approaches can be found in \cite[\S 3]{GT2} and \cite[\S 8]{Gu1}; the strategy
employed here is based on the latter.


It is helpful to change our perspective by switching to the dual handle
decomposition of $R_{i}$. Let $S_{i}$ be a closed collar neighborhood of
$\partial N_{i+1}$ in $R_{i}$, and for each $\left(  n-2\right)  $-handle
$h_{k}^{n-2}$ identified earlier, let $\overline{h}_{k}^{2}$ be its dual,
attached to $S_{i}$. Similarly, for each $\left(  n-3\right)  $-handle
$h_{k}^{n-3}$, let $\overline{h}_{k}^{3}$ be its dual. As is standard, the
attaching and belt spheres of a given handle switch roles in its dual.

Let $T_{i}=S_{i}\bigcup\left(  \overline{h}_{1}^{2}\cup\cdots\cup\overline
{h}_{s}^{2}\cup\overline{h}_{s+1}^{2}\cup\cdots\cup\overline{h}_{r}%
^{2}\right)  $, $\partial_{-}T_{i}=\partial T_{i}-\partial N_{i+1}$, and
$U_{i}$ be a closed collar on $\partial_{-}T_{i}$ in $T_{i}$. Observe that
$R_{i}=T_{i}\bigcup\left(  \overline{h}_{1}^{3}\cup\cdots\cup\overline{h}%
_{t}^{3}\right)  $. See Figure \ref{Fig2-Ri}.%
\begin{figure}
[ptb]
\begin{center}
\includegraphics[
height=2.5771in,
width=4.3344in
]%
{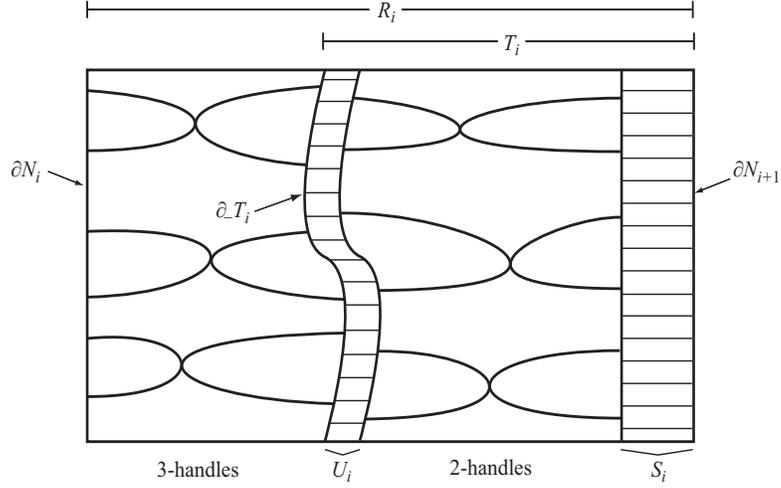}%
\caption{Schematic of $R_{i}$}%
\label{Fig2-Ri}%
\end{center}
\end{figure}

A simplified view of the next step is that we will find a collection of
$3$-handles $\left\{  \overline{k}_{1}^{3},\cdots,\overline{k}_{s}%
^{3}\right\}  $ attached to the left hand boundary of $R_{i}$ and lying in
$R_{i-1}$ so that the collection $\left\{  \Gamma_{j}^{2}\right\}  _{j=1}^{s}$
of attaching spheres of those $3$-handles is algebraically dual to the belt
spheres of $\left\{  \overline{h}_{1}^{2},\cdots,\overline{h}_{s}^{2}\right\}
$ and has trivial algebraic intersection with the belt spheres of $\left\{
\overline{h}_{s+1}^{2},\cdots,\overline{h}_{r}^{2}\right\}  $. Adding those
$3$-handles to the mix, then inverting the handle decomposition again, results
in a cobordism with chain complex
\begin{equation}
0\rightarrow\left\langle h_{1}^{n-2},\cdots,h_{s}^{n-2}\right\rangle
\oplus\left\langle h_{s+1}^{n-2},\cdots,h_{r}^{n-2}\right\rangle
\overset{\partial}{\rightarrow}\left\langle k_{1}^{n-3},\cdots,h_{s}%
^{n-3}\right\rangle \oplus\left\langle h_{1}^{n-3},\cdots,h_{t}^{n-3}%
\right\rangle \rightarrow0 \label{Chain complex-enhanced}%
\end{equation}
in which $\ker\partial=0$ as desired---but with a caveat. Although addition of
the $3$-handles does not change the fundamental group of the cobordism, the
arranged algebraic intersections between the attaching spheres of $\left\{
\overline{k}_{1}^{3},\cdots,\overline{k}_{s}^{3}\right\}  $ and the belt
spheres of the existing $2$-handles are $%
\mathbb{Z}
\left[  \pi_{1}\left(  R_{i}\right)  /L_{i}\right]  $-intersection numbers;
this is the best the hypotheses will allow. Then, to arrive at the desired
conclusion---that we have effectively killed the relative second homology, it
is necessary to switch the coefficient ring to $%
\mathbb{Z}
\left[  \pi_{1}\left(  R_{i}\right)  /L_{i}\right]  $ (in other words, mod out
by $L_{i}$), and reinterpret (\ref{Chain complex-enhanced}) as a $%
\mathbb{Z}
\left[  \pi_{1}\left(  R_{i}\right)  /L_{i}\right]  $-complex. Then, letting
$V_{i}=N_{i}\bigcup\left(  \overline{k}_{1}^{3}\cup\cdots\cup\overline{k}%
_{s}^{3}\right)  $, it follows that: $\pi_{1}\left(  V_{i}\right)  \cong%
\pi_{1}\left(  R_{i}\right)  \cong\pi_{1}\left(  N_{i}\right)  $; $\partial
V_{i}\hookrightarrow V_{i}$ induces a $\pi_{1}$-isomorphism; and $H_{\ast
}(V_{i},\partial V_{i};%
\mathbb{Z}
\left[  \pi_{1}\left(  R_{i}\right)  /L_{i}\right]  )=0$. In other words,
$V_{i}$ is a $\operatorname{mod}\left(  L_{i}\right)  $-homotopy collar.

In order to carry out the above program, we first identify a collection
$\left\{  \Gamma_{j}^{2}\right\}  _{j=1}^{s}$ of pairwise disjoint $2$-spheres
in $\partial_{-}T_{i}$ algebraically dual over $%
\mathbb{Z}
\left[  \pi_{1}\left(  R_{i}\right)  /L_{i}\right]  $ to the collection
$\left\{  \beta_{j}^{n-3}\right\}  _{j=1}^{s}$ of belt spheres of the
$2$-handles $\left\{  \overline{h}_{1}^{2},\cdots,\overline{h}_{s}%
^{2}\right\}  $ and having trivial $%
\mathbb{Z}
\left[  \pi_{1}\left(  R_{i}\right)  /L_{i}\right]  $-intersections with the
belt spheres $\left\{  \beta_{j}^{n-3}\right\}  _{j=s+1}^{r}$of the remaining
$2$-handles $\left\{  \overline{h}_{s+1}^{2},\cdots,\overline{h}_{r}%
^{2}\right\}  $. Keeping in mind that $\pi_{1}\left(  R_{i}\right)  /L_{i}$ is
canonically isomorphic to $\pi_{1}\left(  R_{i+1}\right)  /J_{i+1}$, and using
the hypothesis that $K_{i+1}$ is strongly $J_{i+1}$-perfect, such a collection
$\left\{  \Gamma_{j}^{2}\right\}  _{j=1}^{s}$ exists, as is shown in
\cite[\S 5]{GT3}. By general position, the collection can be made disjoint
from the attaching tubes of the $3$-handles $\left\{  \overline{h}_{1}%
^{3},\cdots,\overline{h}_{t}^{3}\right\}  $, so they may be viewed as lying in
$\partial N_{i}$. If the collection $\left\{  \Gamma_{j}^{2}\right\}
_{j=1}^{s}$ bounds a pairwise disjoint collection of embedded $3$-disks in
$R_{i-1}$, regular neighborhoods of those disks would provide the desired
$3$-handles, and the proof is complete. (The argument from \cite[\S 8]{Gu1}
provides details.)

For $n\geq7$, the issue is just whether the $2$-spheres $\left\{  \Gamma
_{j}^{2}\right\}  _{j=1}^{s}$ contract in $R_{i-1}$. (In dimension $6$, a
special argument is needed to get pairwise disjoint embeddings.)
Contractibility is not guaranteed; but with additional work it can be
arranged. The \textquotedblleft additional work\textquotedblright\ involves
the \emph{spherical alteration}\ of $2$-handles developed in \cite{GT3}. The
idea is to alter the $2$-handles $\left\{  \overline{h}_{1}^{2},\cdots
,\overline{h}_{s}^{2}\right\}  $ in a preplanned manner so that the
correspondingly altered $\left\{  \Gamma_{j}^{2}\right\}  _{j=1}^{s}$ contract
in the new $R_{i-1}$. Along the way it will be necessary to reconstruct the
$3$-handles $\left\{  \overline{h}_{1}^{3},\cdots,\overline{h}_{t}%
^{3}\right\}  $ as well; for later use, let $\left\{  \Theta_{j}^{2}\right\}
_{j=1}^{t}$ denote the attaching spheres of those handles.

All details were carefully laid out in \cite{GT3}, with this application in
mind. The tailor-made lemma, stated in the final section of that paper, is
repeated here.


\begin{lemma}
[{\cite[Lemma 6.1]{GT3}}]\label{Lemma from Spherical Alterations paper}Let
$R^{\prime}\subseteq R$ be a pair of n-manifolds $(n\geq6)$ with a common
boundary component $B$, and suppose there is a subgroup $L^{\prime}$ of
$\ker(\pi_{1}\left(  B\right)  \rightarrow\pi_{1}\left(  R\right)  )$ for
which $K=\ker(\pi_{1}\left(  B\right)  \rightarrow\pi_{1}\left(  R^{\prime
}\right)  )$ is strongly $L^{\prime}$-perfect. Suppose further that there is a
clean submanifold $T\subseteq R^{\prime}$ consisting of a finite collection
$\mathcal{H}^{2}$ of 2-handles in $R^{\prime}$ attached to a collar
neighborhood $S$ of $B$ with $T\hookrightarrow R^{\prime}$ inducing a $\pi
_{1}$-isomorphism (the 2-handles precisely kill the group $K$) and a finite
collection $\left\{  \Theta_{t}^{2}\right\}  $ of pairwise disjoint embedded
2-spheres in $\partial T-B$, each of which contracts in $R^{\prime}$. Then on
any subcollection $\left\{  h_{j}^{2}\right\}  _{j=1}^{k}\subseteq
\mathcal{H}^{2}$ , one may perform spherical alterations to obtain 2-handles
$\left\{  \dot{h}_{j}^{2}\right\}  _{j=1}^{k}$ in $R^{\prime}$ so that in
$\partial\dot{T}-B$ (where $\dot{T}$ is the correspondingly altered version of
$T$) there is a collection of 2-spheres $\left\{  \dot{\Gamma}_{j}%
^{2}\right\}  _{j=1}^{k}$ algebraically dual over $\mathbb{Z}\left[  \pi
_{1}\left(  B\right)  /L^{\prime}\right]  $ to the belt spheres $\left\{
\beta_{j}^{n-3}\right\}  _{j=1}^{k}$ common to $\left\{  h_{j}^{2}\right\}
_{j=1}^{k}$ and $\left\{  \dot{h}_{j}^{2}\right\}  _{j=1}^{k}$ with the
property that each $\dot{\Gamma}_{j}^{2}$ contracts in $R$. Furthermore, each
correspondingly altered 2-sphere $\dot{\Theta}_{t}^{2}$ (now lying in
$\partial\dot{T}-B$) has the same $\mathbb{Z}\left[  \pi_{1}\left(  B\right)
/L^{\prime}\right]  $-intersection number with those belt spheres and with any
other oriented $(n-3)$-manifold lying in both $\partial T-B$ and $\partial
\dot{T}-B$ as did $\Theta_{t}^{2}$. Whereas the 2-spheres $\left\{  {\Theta
}_{t}^{2}\right\}  $ each contracted in $R^{\prime}$, the $\dot{\Theta}%
_{t}^{2}$ each contract in $R$.
\end{lemma}

Apply Lemma \ref{Lemma from Spherical Alterations paper} to the current setup,
with the following substitutions:
\[%
\begin{array}
[c]{llll}%
\underline{\mathbf{Lemma}\text{ \ref{Lemma from Spherical Alterations paper}}}
&  & \hspace*{0.05in} & \underline{\mathbf{Current\ situation}\text{ }}\\
R^{\prime} & \leftrightarrow &  & R_{i}\\
R & \leftrightarrow &  & R_{i}\bigcup R_{i-1}\\
B & \leftrightarrow &  & \partial N_{i+1}\\
\mathcal{H}^{2} & \leftrightarrow &  & \left\{  \overline{h}_{1}^{2}%
,\cdots,\overline{h}_{s}^{2},\overline{h}_{s+1}^{2},\cdots,\overline{h}%
_{r}^{2}\right\} \\
L^{\prime} & \leftrightarrow &  & J_{i+1}=\lambda_{i+1}^{-1}\left(
L_{i}\right) \\
T & \leftrightarrow &  & T_{i}=S_{i}\bigcup\left(  \overline{h}_{1}^{2}%
\cup\cdots\cup\overline{h}_{s}^{2}\cup\overline{h}_{s+1}^{2}\cup\cdots
\cup\overline{h}_{r}^{2}\right) \\
k\in%
\mathbb{Z}%
& \leftrightarrow &  & s\in%
\mathbb{Z}%
\\
\left\{  h_{j}^{2}\right\}  _{j=1}^{k} & \leftrightarrow &  & \left\{
\overline{h}_{j}^{2}\right\}  _{j=1}^{s}\\
\left\{  \Gamma_{j}^{2}\right\}  _{j=1}^{k} & \leftrightarrow &  & \left\{
\Gamma_{j}^{2}\right\}  _{j=1}^{s}\\
\left\{  \Theta_{t}^{2}\right\}  & \leftrightarrow &  & \left\{  \Theta
_{j}^{2}\right\}  _{j=1}^{t}%
\end{array}
\]

After applying this lemma, the collection $\left\{  \overline{h}_{j}%
^{2}\right\}  _{j=1}^{s}$ is replaced by altered versions $\left\{
\overset{\cdot}{\bar{h}_{j}^{2}}\right\}  _{j=1}^{s}$ and the original
collection $\left\{  \overline{h}_{j}^{2}\right\}  _{j=s+1}^{r}$ is retained.
Let
\[
\dot{T}_{i}=S_{i}\bigcup\left(  \overset{\cdot}{\bar{h}_{1}^{2}}\cup\cdots
\cup\overset{\cdot}{\bar{h}_{s}^{2}}\cup\overline{h}_{s+1}^{2}\cup\cdots
\cup\overline{h}_{r}^{2}\right)
\]
and $\partial_{-}\dot{T}_{i}=\partial\dot{T}_{i}-\partial N_{i+1}$. The
collections $\left\{  \Gamma_{j}^{2}\right\}  _{j=1}^{s}$ and $\left\{
\Theta_{j}^{2}\right\}  _{j=1}^{t}$ are replaced by altered versions $\left\{
\dot{\Gamma}_{j}^{2}\right\}  _{j=1}^{s}$ and $\left\{  \dot{\Theta}_{j}%
^{2}\right\}  _{j=1}^{t}$ which lie in $\partial_{-}\dot{T}_{i}$ and contract
in $\overline{R_{i}\cup R_{i-1}-\dot{T}}$. The original $3$-handles $\left\{
\overline{h}_{j}^{3}\right\}  _{j=1}^{t}$ must be discarded since their
attaching tubes have been disrupted; replacements will be constructed shortly.
When $n\geq7$, use general position to choose pairwise disjoint collection of
properly embedded $3$-disks in $\overline{R_{i}\cup R_{i-1}-\dot{T}}$ with
boundaries corresponding to the $2$-spheres $\left\{  \dot{\Gamma}_{j}%
^{2}\right\}  _{j=1}^{k}\cup\left\{  \dot{\Theta}_{t}^{2}\right\}  $. Those
$3$-disks may be thickened to $3$-handles by taking regular neighborhoods.
With all of these handles finally in place, the argument described earlier
completes the proof. When $n=6$, the same is true, but the $\pi$-$\pi$
argument used in \cite[Thms. 4.2 \& 5.3]{GT3} is needed in order to find
pairwise disjoint embedded $3$-disks.




\end{proof}

\begin{remark}
\emph{In reality, we have shown a stronger result than what is stated in
Theorem \ref{NPCT technical version}. Specifically, the near pseudo-collar
structures obtained are as close to actual pseudo-collars as the augmentation
is to the trivial augmentation. For example, if }$\left\{  L_{i}\right\}
$\emph{ is the trivial augmentation, the above argument contains an
alternative proof of the main result of \cite{GT2} (stated here as Theorem
\ref{PCT technical version}). More generally, if }$\left\{  L_{i}\right\}
$\emph{ lies somewhere between the trivial augmentation and the standard
augmentation, then a near pseudo-collar structure on }$M^{n}$\emph{ can be
chosen to reflect that augmentation.}
\end{remark}



\section{The Examples: Proof of Theorem
\ref{Theorem: Existence of counterexamples}}

\subsection{Introduction to the examples}

The main examples of \cite{GT1}, described here in Example \ref{Example: 1},
proved the existence of (absolutely) inward tame open manifolds that are not
pseudo-collarable. In this section we construct open manifolds that are
absolutely inward tame but not nearly pseudo-collarable. Since the examples
from \cite{GT1} are nearly pseudo-collarable, the new examples fill a gap in
the spectrum of known end structures.

The examples of \cite{GT1} began with algebra. The main theorems of that paper
showed that all inward tame open manifolds have pro-finitely generated,
semistable fundamental group, and stable $%
\mathbb{Z}
$-homology, at infinity. The missing ingredient for detecting a pseudo-collar
structure was $\mathcal{P}$-semistability. With that knowledge, an inverse
sequence of groups satisfying the necessary properties, but failing
$\mathcal{P}$-semistability, became the blueprint for an example. A nontrivial
handle-theoretic strategy was needed to realize the examples, but the heart of
the matter was the group theory.

A similar story plays out here. We will begin with an inverse sequence of
finitely presented groups with surjective bonding maps that become
isomorphisms upon abelianization; but this time, in light of Theorems
\ref{Theorem: inward tame implies AP-semistability} and
\ref{Theorem: NPC Characterization--simple version}, we want an $\mathcal{AP}%
$-semistable sequence that is not $\mathcal{SAP}$-semistable. The first step
is to identify such a sequence.

Let $\mathbb{F}_{3}=\left\langle a_{1},a_{2},a_{3}\mid\ \right\rangle $, the
free group on the three generators; $r_{1,1}=\left[  a_{2},a_{3}\right]  $,
$r_{1,2}=\left[  a_{1},a_{3}\right]  $, and $r_{1,3}=\left[  a_{1}%
,a_{2}\right]  $; $\mathbb{A}_{1}=\operatorname*{ncl}\left(  \{r_{1,1}%
,r_{1,2},r_{1,3}\},\mathbb{F}_{3}\right)  $; and $G_{1}=\mathbb{F}%
_{3}/\mathbb{A}_{1}$. Notice that $\mathbb{A}_{1}$ is precisely the commutator
subgroup $\left[  \mathbb{F}_{3},\mathbb{F}_{3}\right]  $, so $G_{1}\cong%
\mathbb{Z}
\oplus%
\mathbb{Z}
\oplus%
\mathbb{Z}
$.

Let $r_{2,1}=\left[  r_{1,2,},r_{1,3}\right]  $, $r_{2,2}=\left[
r_{1,1},r_{1,3}\right]  $, and $r_{1,3}=\left[  r_{1,1},r_{1,2}\right]  $;
$\mathbb{A}_{2}=\operatorname*{ncl}\left(  \{r_{2,1},r_{2,2},r_{2,3}%
\},\mathbb{F}_{3}\right)  $; and $G_{2}=\mathbb{F}_{3}/\mathbb{A}_{2}$. Since
$\mathbb{A}_{2}\leq\mathbb{A}_{1}$, there is an induced epimorphism
$G_{1}\overset{\lambda_{2}}{\longleftarrow}G_{2}$ which abelianizes to the
identity map on $%
\mathbb{Z}
\oplus%
\mathbb{Z}
\oplus%
\mathbb{Z}
$.

Continue inductively, letting $r_{i+1,1}=\left[  r_{i,2},r_{i,3}\right]  $,
$r_{i+1,2}=\left[  r_{i,1},r_{i,3}\right]  $, and $r_{i+1,3}=\left[
r_{i,1},r_{i,2}\right]  $; $\mathbb{A}_{i+1}\allowbreak=\allowbreak
\operatorname*{ncl}\left(  \{r_{i+1,1},r_{i+1,2},r_{i+1,3}\},\mathbb{F}%
_{3}\right)  $; and $G_{i+1}=\mathbb{F}_{3}/\mathbb{A}_{i+1}$. The result is a
nested sequence of normal subgroups of $\mathbb{F}_{3}$, $\mathbb{A}_{1}%
\geq\mathbb{A}_{2}\geq\mathbb{A}_{3}\geq\cdots$, and a corresponding inverse
sequence of quotient groups
\begin{equation}
G_{1}\overset{\lambda_{2}}{\twoheadleftarrow}G_{2}\overset{\lambda_{3}%
}{\twoheadleftarrow}G_{3}\overset{\lambda_{4}}{\twoheadleftarrow}%
\cdots\label{Chosen inverse sequence}%
\end{equation}
which abelianizes to the constant inverse sequence
\[%
\mathbb{Z}
^{3}\overset{\operatorname*{id}}{\longleftarrow}%
\mathbb{Z}
^{3}\overset{\operatorname*{id}}{\longleftarrow}%
\mathbb{Z}
^{3}\overset{\operatorname*{id}}{\longleftarrow}\cdots.
\]
A more delicate motivation for our choices is the following: For each $i>1$,
$\ker\lambda_{i}\allowbreak=\allowbreak\operatorname*{ncl}\left(
\{r_{i-1,1},r_{i-1,2},r_{i-1,3}\},G_{i}\right)  $; similarly, for each $i>2$,
\[
\ker(\lambda_{i-1}\lambda_{i})=\operatorname*{ncl}\left(  \{r_{i-2,1}%
,r_{i-2,2},r_{i-2,3}\},G_{i}\right)  \text{.}%
\]
Moreover, since the elements of $\{r_{i-1,1},r_{i-1,2},r_{i-1,3}\}$ are
precisely the commutators of the elements of $\{r_{i-2,1},r_{i-2,2}%
,r_{i-2,3}\}$,
\[
\ker\left(  \lambda_{i}\right)  \leq\left[  \ker\left(  \lambda_{i-1}%
\lambda_{i}\right)  ,\ker\left(  \lambda_{i-1}\lambda_{i}\right)  \right]  .
\]
So, for the standard augmentation, $L_{i}=\ker\lambda_{i}$,
(\ref{Chosen inverse sequence}) is $\left\{  L_{i}\right\}  $-perfect, hence,
$\mathcal{AP}$-semistable.

Two tasks remain:

\begin{itemize}
\item Prove that (\ref{Chosen inverse sequence}) is not $\mathcal{SAP}%
$-semistable, and

\item Construct 1-ended absolutely inward tame open manifolds with fundamental
groups at infinity representable by (\ref{Chosen inverse sequence}).
\end{itemize}

\noindent Since these tasks are independent, the ordering of the following two
\ subsections is arbitrary.

\subsection{ The Sequence \ref{Chosen inverse sequence} is not $\mathcal{SAP}%
$-semistable}

Let $\mathbb{F}_{n}=\left\langle a_{1},\cdots,a_{n}\mid\ \right\rangle $, the
free group on $n$ generators. We will exploit two standard constructions from
group theory. The \emph{derived series of } $\mathbb{F}_{n}$ is defined by
$\mathbb{F}_{n}^{(0)}=\mathbb{F}_{n}$ and $\mathbb{F}_{n}^{(k+1)}=\left[
\mathbb{F}_{n}^{\left(  k\right)  },\mathbb{F}_{n}^{\left(  k\right)
}\right]  $ for $k\geq0$. The \emph{lower central series of} $\mathbb{F}$ is
given by $\left(  \mathbb{F}_{n}\right)  _{1}=\mathbb{F}_{n}$ and then
$\left(  \mathbb{F}_{n}\right)  _{k+1}=\left[  \left(  \mathbb{F}_{n}\right)
_{k},\mathbb{F}_{n}\right]  $ for $k\geq0$. By inspection $\mathbb{F}%
_{n}^{(k+1)}\leq\mathbb{F}_{n}^{\left(  k\right)  }$, $\left(  \mathbb{F}%
_{n}\right)  _{k+1}\leq\left(  \mathbb{F}_{n}\right)  _{k}$, and
$\mathbb{F}_{n}^{(k)}\leq\left(  \mathbb{F}_{n}\right)  _{k+1}$ for all $k$. A
well-known fact, similar in spirit to our goal in this subsection, is that
$\cap_{k=0}^{\infty}\mathbb{F}_{n}^{\left(  k\right)  }=\{1\}=\cap
_{k=1}^{\infty}\left(  \mathbb{F}_{n}\right)  _{k}$.

The following representation of $\mathbb{F}_{n}$ was discovered by Magnus; our
general reference is \cite{LS}.

\begin{proposition}
\label{Prop: representation}~\cite[Proposition 10.1]{LS} Let $\mathcal{P}_{n}$
be the non-commuting power series ring in indeterminates $\left\{  x_{1}%
,x_{2},\cdots,x_{n}\right\}  $ with $x_{j}^{2}=0$ for $j=1,2,\cdots,n$. Then
the function $\beta\left(  a_{j}\right)  =1+x_{j}$ $\left(  j=1,2,\cdots
,n\right)  $ induces a faithful representation of $\mathbb{F}_{n}$ into
$\mathcal{P}_{n}^{\ast}$, the multiplicative group of units of $\mathcal{P}%
_{n}$.
\end{proposition}

In $\mathcal{P}_{n}$, the \emph{fundamental ideal} $\Delta$ is the kernel of
the homomorphism $\rho:\mathcal{P}_{n}\rightarrow{\mathbb{Z}}$ that takes each
$x_{j}$ to 0. The elements of $\Delta$ are all sums of the form $\sum_{\nu
=1}^{\infty}\pi_{\nu}$ where each $\pi_{\nu}$ is a homogeneous polynomial of
degree at least one. Consequently, for any positive integer $k$ the ideal
$\Delta^{k}$ is made of all sums of the form $\sum_{\nu=1}^{\infty}\pi_{\nu}$
where each $\pi_{\nu}$ is a homogeneous polynomial of degree at least $k$.

The next proposition and lemma are useful for monitoring the location of
commutators in a group.

\begin{proposition}
\label{Prop: commutator}~\cite[Proposition 10.2]{LS} Let $\beta:\mathbb{F}%
_{n}\rightarrow\mathcal{P}^{\ast}$ be the representation given above. If
$w_{1},w_{2}\in\mathbb{F}_{n}$ such that $\beta\left(  w_{1}\right)
-1\in\Delta^{r}$ and $\beta\left(  w_{2}\right)  -1\in\Delta^{s}$, then
$\beta\left(  \left[  w_{1},w_{2}\right]  \right)  -1\in\Delta^{r+s}$.
\end{proposition}

By applying Proposition \ref{Prop: commutator} inductively, we obtain the
following useful facts.

\begin{lemma}
\label{Lemma: commutator}For all integers $n,i\geq1$,

\begin{enumerate}
\item $\left\{  \beta(w)-1\mid w\in\mathbb{F}_{n}^{\left(  i\right)
}\right\}  \subseteq\Delta^{2^{i}},$

\item $\left\{  \beta(w)-1\mid w\in\left(  \mathbb{F}_{n}\right)
_{i}\right\}  \subseteq\Delta^{i},$

\item $\bigcap_{k=1}^{\infty}\Delta^{k}=0$, and

\item $\bigcap_{k=1}^{\infty}\mathbb{F}_{n}^{\left(  k\right)  }=\left\{
1\right\}  =\bigcap_{k=1}^{\infty}(\mathbb{F}_{n})_{k}$.
\end{enumerate}
\end{lemma}


We now focus our attention on $\mathbb{F}_{3}$ and its subgroups
$\mathbb{A}_{i}=\operatorname*{ncl}\left(  \left\{  r_{i,1},r_{i,2}%
,r_{i,3}\right\}  ,\mathbb{F}_{3}\right)  $, as defined earlier.

\begin{lemma}
\label{Lemma: topology prelim}For each $k\geq1$ and $j\in\left\{
1,2,3\right\}  $,

\begin{enumerate}
\item $r_{k,j}$ is a member of at least one free basis for $\mathbb{F}%
_{3}^{\left(  k\right)  }$, and \ 

\item $r_{k,j}\in\mathbb{F}_{3}^{\left(  k\right)  }-\mathbb{F}_{3}^{\left(
k+1\right)  }$.
\end{enumerate}
\end{lemma}

\begin{proof}
Assertion (1) can be obtained from an inductive argument using Schreier
systems. A model argument can be found in \cite[Example 8.1]{Ma}.

Assertion (2) follows from (1), since the quotient map $\mathbb{F}_{3}%
^{k}\rightarrow\mathbb{F}_{3}^{k}/\mathbb{F}_{3}^{k+1}$ is the abelianization
of $\mathbb{F}_{3}^{k}$.
\end{proof}

Since $\mathbb{A}_{i}\leq\mathbb{F}_{3}^{\left(  i\right)  }$, the following
is an easy consequence of Lemmas \ref{Lemma: commutator} and
\ref{Lemma: topology prelim}.

\begin{lemma}
\label{Lemma: specialized commutators}For each $i\geq1$ and $j\in\{1,2,3\}$,

\begin{enumerate}
\item $\beta(r_{i,j})-1\neq0$, and

\item $\left\{  \beta(h)-1\mid h\in\mathbb{A}_{i}\right\}  \subseteq
\Delta^{2^{i}}$.
\end{enumerate}
\end{lemma}

The definitions of derived and lower central series are clearly applicable to
arbitrary groups. To expand those notions further, the following definition is
useful. For $H\trianglelefteq G$, let $\Omega_{1}\left(  H,G\right)  =H$ and
$\Omega_{k}\left(  H,G\right)  =\left[  \Omega_{k-1}\left(  H,G\right)
,G\right]  $ for $k>1$. By normality, $H=\Omega_{1}\left(  H,G\right)
\geq\Omega_{2}\left(  H,G\right)  \geq\Omega_{3}\left(  H,G\right)  \geq
\cdots$. When $H$ is strongly $G$-perfect, $\Omega_{k}\left(  H,G\right)  =H$
for all $k$.

\begin{proposition}
\label{Prop: omega groups get small}For each $i\geq1$, there exists $p_{i}>0$
and $q_{i}\geq p_{i}$ such that

\begin{enumerate}
\item for each $j\in\left\{  1,2,3\right\}  $, $\beta(r_{i,j})-1\notin
\Delta^{2^{i}+p_{i}}$, and

\item $\{\beta\left(  w\right)  -1\mid w\in\Omega_{q_{i}}\left(
\mathbb{A}_{i},\mathbb{F}_{3}\right)  \}\subseteq\Delta^{2^{i}+p_{i}}$.
\end{enumerate}
\end{proposition}

\begin{proof}
Let $i$ be fixed. Existence of $p_{i}$ follows from item (3) of Lemma
\ref{Lemma: commutator}. Existence of $q_{i}$ may be obtained from an
inductive application of \ref{Prop: commutator}.
\end{proof}


We shift focus one more time; from $\mathbb{F}_{3}$ and its subgroups to the
quotient groups $G_{i}=\mathbb{F}_{3}/\mathbb{A}_{i}$ and their subgroups. In
doing so, we will allow a word in the generators of $\mathbb{F}_{3}$ to
represent both an element of $\mathbb{F}_{3}$ and the corresponding element of
a $G_{i}$. For example, recalling that $\lambda_{i+1,j}=\lambda_{i+1}%
\circ\cdots\circ\lambda_{j}:G_{j}\rightarrow G_{i}$, we say $\ker\left(
\lambda_{i+1,j}\right)  =\operatorname*{ncl}\left(  \left\{  r_{i,1}%
,r_{i,2},r_{i,3}\right\}  ,G_{j}\right)  $.

The following is simple but useful.

\begin{lemma}
\label{Lemma: images of omega}Suppose $\lambda:G\rightarrow G^{\prime}$ is a
surjective homomorphism, $H\trianglelefteq G$, and $q\geq0$. Then
$\lambda\left(  \Omega_{q}\left(  H,G\right)  \right)  =\Omega_{q}\left(
\lambda\left(  H\right)  ,G^{\prime}\right)  $.
\end{lemma}

Lemma \ref{Lemma: images of omega} ensures that, for each $i<k$ and all
$q\geq0$, the quotient maps $\mathbb{F}_{3}\twoheadrightarrow G_{k}$ restrict
to epimorphisms
\begin{equation}
\Omega_{q}(\mathbb{A}_{i},\mathbb{F}_{3})\twoheadrightarrow\Omega_{q}\left(
\operatorname*{ncl}\left(  \left\{  r_{i,1},r_{i,2},r_{i,3}\right\}
,G_{k}\right)  ,G_{k}\right)  \text{.}%
\end{equation}
\label{equation: key epimorphism}

\begin{proposition}
\label{Prop: no omega} For $p_{i}$ and $q_{i}$ as chosen in Proposition
\ref{Prop: omega groups get small}, and each $j\in\left\{  1,2,3\right\}  $,
$r_{i,j}\notin\Omega_{q_{i}}\left(  \ker\left(  \lambda_{i+1,k}\right)
,G_{k}\right)  $ whenever $2^{k}\geq2^{i}+p_{i}$.
\end{proposition}

\begin{proof}
Suppose $r_{i,j}\in\Omega_{q_{i}}\left(  \ker\left(  \lambda_{i+1,k}\right)
,G_{k}\right)  =\Omega_{q_{i}}\left(  \operatorname*{ncl}\left(  \left\{
r_{i,1},r_{i,2},r_{i,3}\right\}  ,G_{k}\right)  ,G_{k}\right)  $. Surjection
(\ref{equation: key epimorphism}) provides a $w\in\Omega_{q_{i}}\left(
\mathbb{A}_{i},\mathbb{F}_{3}\right)  $ with cosets $\mathbb{A}_{k}\cdot
r_{i,j}=\mathbb{A}_{k}\cdot w$. Consequently, there is an $h\in\mathbb{A}_{k}$
with $r_{i,j}=hw$ in $\mathbb{F}_{3}$. Then%
\[%
\begin{array}
[c]{rcl}%
\beta(r_{i,j})-1 & = & \beta(h)\beta\left(  w\right)  -1\\
&  & \beta(h)\beta\left(  w\right)  -\beta(h)+\beta(h)-1\\
&  & \beta(h)\left(  \beta\left(  w\right)  -1\right)  +\left(  \beta
(h)-1\right)
\end{array}
\]
Since $\beta\left(  w\right)  -1\in\Delta^{2^{i}+p_{i}}$ and $\beta
(h)-1\in\Delta^{2^{k}}\subseteq\Delta^{2^{i}+p_{i}}$, then $\beta
(r_{i,j})-1\in\Delta^{2^{i}+p_{i}}$, violating the choice of $p_{i}$.
\end{proof}

We are now ready for the main result of this subsection..

\begin{theorem}
\label{Theorem: example} The inverse sequence $\left\{  G_{i},\lambda
_{i}\right\}  _{i=0}^{\infty}$ is not $\mathcal{SAP}$-semistable. In fact,
$\left\{  G_{i},\lambda_{i}\right\}  _{i=0}^{\infty}$ is not pro-isomorphic to
any inverse sequence $\left\{  H_{i},\mu_{i}\right\}  $ of surjections that
satisfies the strong $\left\{  H_{i}\right\}  $-perfectness property.
\end{theorem}

\begin{proof}
We proceed directly to the stronger assertion. Suppose $\left\{  G_{i}%
,\lambda_{i}\right\}  $ is pro-isomorphic to an inverse sequence $\left\{
H_{i},\mu_{i}\right\}  $ of surjections, that is strongly $\left\{
H_{i}\right\}  $-perfect; in other words, $\ker\mu_{i}=[\ker\mu_{i},H_{i}]$
for all $i$.

By Proposition \ref{Prop: nearly sub}, each subsequence of $\left\{  H_{i}%
,\mu_{i}\right\}  $ satisfies the same essential property, so by our
assumption, $\left\{  G_{i},\lambda_{i}\right\}  $ contains a subsequence that
fits into a commutative diagram of the following form:%

\[
\begin{diagram}
G_{i_{0}} & & \lTo^{\lambda_{i_{0}+1,i_{1}}} & & G_{i_{1}} & &
\lTo^{\lambda_{i_{1}+1,i_{2}}} & & G_{i_{2}} & & \lTo^{\lambda_{i_{2}+1,i_{3}}}& & G_{i_{3}}& \cdots\\
& \luTo^{u_{0}} & & \ldTo^{d_{1}} & & \luTo^{u_{1}} & & \ldTo^{d_{2}}   & & \luTo^{u_{2}} & & \ldTo^{d_{3}}   &\\
& & H_{0} & & \lOnto^{\mu_{1}} & & H_{1} & &
\lOnto^{\mu_{2}}& & H_{2} & & \lOnto^{\mu_{3}} & & \cdots
\end{diagram}
\]


Passing to a further subsequence if necessary, we may assume that $2^{i_{n}%
}\geq2^{i_{n-1}}+p_{i_{n-1}}$ for all $n$.

By Lemma \ref{Lemma: topology prelim}, $1\neq r_{i_{1},j}\in
\operatorname*{ker}\left(  \lambda_{i_{1}+1,i_{2}}\right)  \leq G_{i_{2}}$.
Choose $\alpha^{\prime}\in H_{2}$ with $u_{2}(\alpha^{\prime})=r_{i_{1},j}$.
Then, $\alpha^{\prime}\in\operatorname*{ker}\left(  \mu_{1,2}\right)  $, and
consequently $\alpha^{\prime}\in\left[  \operatorname*{ker}\left(  \mu
_{1,2}\right)  ,H_{2}\right]  $, since $\operatorname*{ker}\left(  \mu
_{1,2}\right)  $ is strongly $H_{2}$-perfect (again using Proposition
\ref{Prop: nearly sub}). Therefore $\alpha^{\prime}\in\Omega_{q}\left(
\operatorname*{ker}\left(  \mu_{1,2}\right)  ,H_{2}\right)  $ for all $q$.
Moreover, since $u_{2}\left(  \operatorname*{ker}\left(  \mu_{1,2}\right)
\right)  \subseteq\operatorname*{ker}\left(  \lambda_{i_{0}+1,i_{2}}\right)
$,
\[
r_{i_{1},j}=u_{2}(\alpha^{\prime})\in\Omega_{q}\left(  u_{2}\left(
\operatorname*{ker}\left(  \mu_{1,2}\right)  \right)  ,G_{i_{2}}\right)
\subseteq\Omega_{q}\left(  \operatorname*{ker}\left(  \lambda_{i_{0}+1,i_{2}%
}\right)  ,G_{i_{2}}\right)
\]
for all $q$, thereby contradicting Proposition \ref{Prop: no omega}.


\end{proof}

\subsection{Construction of the examples}

The goal of this subsection is to construct, for each $n\geq6$, a 1-ended open
manifold $M^{n}$ that is absolutely inward tame and has fundamental group at
infinity represented by the inverse sequence (\ref{Chosen inverse sequence}).
By Theorem \ref{Theorem: NPC Characterization--simple version} or
\ref{NPCT technical version}, such an example fails to be nearly
pseudo-collarable, thus completing the proof of Theorem
\ref{Theorem: Existence of counterexamples}.

\subsubsection{Overview}

We will construct $M^{n}$ as a countable union of codimension 0 submanifolds
\[
M^{n}=C_{1}\cup A_{1}\cup A_{2}\cup A_{3}\cup\medskip\cdots
\]
where $C_{1}$ is a compact \textquotedblleft core\textquotedblright\ and
$\left\{  \left(  A_{i},\Gamma_{i},\Gamma_{i+1}\right)  \right\}  $ is a
sequence of compact cobordisms between closed connected $\left(  n-1\right)
$-manifolds where $A_{i}\cap A_{i+1}=\Gamma_{i+1}$ for each $i\geq1$, and
$\partial C_{1}=\Gamma_{1}$. Letting
\[
N_{i}=A_{i}\cup A_{i+1}\cup A_{i+2}\cup\medskip\cdots
\]
gives a preferred end structure $\left\{  N_{i}\right\}  $ with $\partial
N_{i}=\Gamma_{i}$ for each $i$. See Figure \ref{Fig3-Union of cobordisms}.%
\begin{figure}
[ptb]
\begin{center}
\includegraphics[
trim=0.347738in 0.686478in 0.000000in 0.254219in,
height=2.9187in,
width=5.047in
]%
{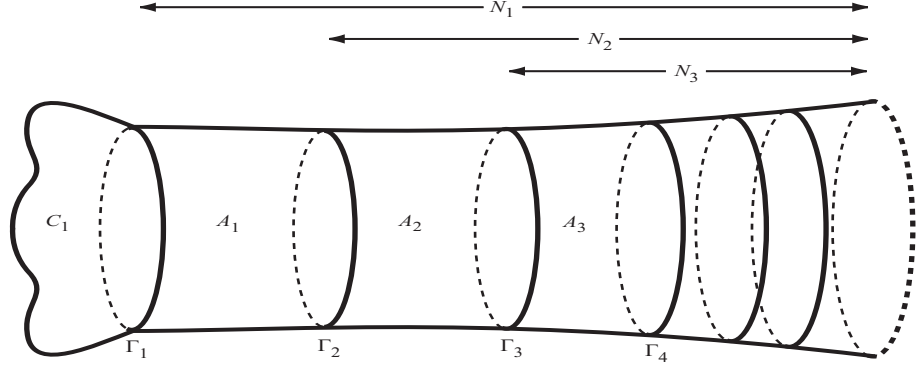}%
\caption{$M^{n}=C_{1}\cup A_{1}\cup A_{2}\cup A_{3}\cup\cdots$}%
\label{Fig3-Union of cobordisms}%
\end{center}
\end{figure}

So that $\operatorname*{pro}$-$\pi_{1}\left(  \varepsilon\left(  M^{n}\right)
\right)  $ is represented by (\ref{Chosen inverse sequence}), the $A_{i}$ will
be constructed to satisfy:\smallskip

\begin{enumerate}
\item[(a)] For all $i\geq1$, $\pi_{1}(\Gamma_{i},p_{i})\cong G_{i}$ and
$\Gamma_{i}\hookrightarrow A_{i}$ induces a $\pi_{1}$-isomorphism, and
\smallskip

\item[(b)] The isomorphism between $\pi_{1}(\Gamma_{i},p_{i})$ and $G_{i}$ may
be chosen so that the following diagram commutes:%

\[%
\begin{array}
[c]{ccccc}%
G_{i} &  & \overset{\lambda_{i+1}}{\longleftarrow} &  & G_{i+1}\\
\downarrow\cong &  &  &  & \downarrow\cong\\
\pi_{1}(\Gamma_{i},p_{i}) & \overset{\cong}{\longleftarrow} & \pi_{1}\left(
A_{i},p_{i}\right)  & \overset{\psi_{i+1}}{\longleftarrow} & \pi_{1}%
(\Gamma_{i+1},p_{i+1})\medskip
\end{array}
\]
\medskip Here $\psi_{i+1}$ is the composition%
\[
\pi_{1}\left(  A_{i},p_{i}\right)  \overset{\widehat{\rho}_{i}}{\leftarrow}%
\pi_{1}\left(  A_{i},p_{i+1}\right)  \overset{\iota_{i+1}}{\longleftarrow}%
\pi_{1}\left(  \Gamma_{i+1},p_{i+1}\right)
\]
where $\iota_{i+1}$ is induced by inclusion and $\widehat{\rho}_{i}$ is a
change-of-basepoint isomorphism\ with respect to a path $\rho_{i}$ in $A_{i}$
between $p_{i}$ and $p_{i+1}$.
\end{enumerate}

From there it follows from Van Kampen's theorem that each $\Gamma_{i}=\partial
N_{i}\hookrightarrow N_{i}$ induces a\ $\pi_{1}$-isomorphism, so by repeated
application of (a) and (b), the inverse sequence%
\[
\pi_{1}\left(  N_{1},p_{1}\right)  \overset{\mu_{2}}{\longleftarrow}\pi
_{1}\left(  N_{2},p_{2}\right)  \overset{\mu_{3}}{\longleftarrow}\pi
_{1}\left(  N_{3},p_{3}\right)  \overset{\mu_{4}}{\longleftarrow}%
\]
is isomorphic to (\ref{Chosen inverse sequence}).

It will be also be shown that each $N_{i}$ has finite homotopy type; so
$M^{n}$ is absolutely inward tame. That argument requires specific details of
the construction; it will be presented later.

\subsubsection{Details of the construction}

Recall that a $p$-handle $h^{p}$ attached to an $n$-manifold $P^{n}$ and a
$\left(  p+1\right)  $-handle $h^{p+1}$ attached to $P^{n}\cup h^{p}$ form a
\emph{complementary pair} if the attaching sphere of $h^{p+1}$ intersects the
belt sphere of $h^{p}$ transversely in a single point. In that case $P^{n}\cup
h^{p}\cup h^{p+1}\approx P^{n}$; moreover, we may arrange (by an isotopy of
the attaching sphere of $h^{p+1}$) that $P^{n}\cap(h^{p}\cup h^{p+1})$ is an
$\left(  n-1\right)  $-ball in $\partial P^{n}$. Conversely, for any ball
$B^{n-1}\subseteq\partial P^{n}$, one may introduce a pair of complementary
handles $P^{n}\cup h^{p}\cup h^{p+1}$ so that $P^{n}\cap(h^{p}\cup
h^{p+1})=B^{n-1}$. We call $\left(  h^{p},h^{p+1}\right)  $ a \emph{trivial
handle pair.} Note that the difference between a complementary pair and
trivial pair is just a matter of perspective. In general, we say that $h^{p}$
is \emph{attached trivially} to $P^{n}$ if it is possible to attach an
$h^{p+1}$ so that $\left(  h^{p},h^{p+1}\right)  $ is a complementary
pair.\medskip

After a preliminary step where we construct the core manifold $C_{1}$, our
proof proceeds inductively. At the $i^{\text{th}}$ stage we construct the
cobordism $\left(  A_{i},\Gamma_{i},\Gamma_{i+1}\right)  $, along with a
compact manifold $C_{i+1}$ with $\partial C_{i+1}=\Gamma_{i+1}$, to be used in
the following stage. Throughout the construction, we abuse notation slightly
by letting $\partial C_{i}\times\left[  0,\varepsilon\right]  $ denote a small
regular neighborhood of $\partial C_{i}$ in $C_{i}$ and $\Gamma_{i}%
\times\left[  0,\varepsilon\right]  $ to denote a small regular neighborhood
of $\Gamma_{i}$ in $A_{i}$\medskip

\noindent\textbf{Step 0. }(Preliminaries)

Let $C_{0}$ be the $n$-manifold obtained by attaching three orientable
1-handles $\left\{  h_{0,j}^{1}\right\}  _{j=1}^{3}$ to the $n$-ball $B^{n}$.
Choose a basepoint $p_{0}\in\partial C_{0}$ and let $a_{1},a_{2}$, and $a_{3}$
be be embedded loops in $\partial C_{0}$ intersecting only at $p_{0}$. Abuse
notation slightly by writing
\[
\pi_{1}\left(  \partial C_{0}\right)  =\pi_{1}\left(  C_{0}\right)
=\left\langle a_{1},a_{2},a_{3}\mid\ \right\rangle .
\]
A convenient way to arrange that the 1-handles are orientable is by attaching
three trivial $\left(  1,2\right)  $-handle pairs $\left\{  h_{0,j}%
^{1},h_{0,j}^{2}\right\}  _{j=1}^{3}$, then discarding the $2$-handles.

Recall that $G_{1}=\left\langle a_{1},a_{2},a_{3}\mid r_{1,1},r_{1,2}%
,r_{1,3}\right\rangle $ where $r_{1,1}=\left[  a_{2},a_{3}\right]  $,
$r_{1,2}=\left[  a_{1},a_{3}\right]  $, and $r_{1,3}=\left[  a_{1}%
,a_{2}\right]  $. Attach a trio of $2$-handles $\left\{  h_{1,j}^{2}\right\}
_{j=1}^{3}$ to $C_{0}$, where $h_{1,j}^{2}$ has attaching circle $r_{1,j}$.
Choose the framings of these handles so that, if the $2$-handles $\left\{
h_{0,j}^{2}\right\}  _{j=1}^{3}$ were added back in, then $\left\{
h_{1,j}^{2}\right\}  _{j=1}^{3}$ would be trivially attached (to an $n$-ball).
Let
\[
C_{1}=C_{0}\cup h_{1,1}^{2}\cup h_{1,2}^{2}\cup h_{1,3}^{2}%
\]
and note that $\pi_{1}\left(  C_{1}\right)  \cong\pi_{1}\left(  \partial
C_{1}\right)  \cong G_{1}$.\medskip

\noindent\textbf{Step 1. }(Constructing $A_{1}$ and $C_{2}$)

Attach three trivial $\left(  2,3\right)  $-handle pairs to $C_{1}$, disjoint
from the existing handles, then perform handle slides on each of the trivial
2-handles (over the handles $\{h_{1,j}^{2}\}_{j=1}^{3}$) so that the resulting
2-handles $h_{2,1}^{2}$, $h_{2,2}^{2}$ and $h_{2,3}^{2}$ have attaching
circles spelling out the words $r_{2,1}$, $r_{2,2}$ and $r_{2,3}$,
respectively. This is possible since each $r_{2,k}$ can be viewed as a product
of the loops $\left\{  r_{1,j}\right\}  _{j=1}^{3}$ and their inverses, which
are the attaching circles of $\{h_{1,j}^{2}\}_{j=1}^{3}$. Sliding a 2-handle
over $h_{1,j}^{2}$ inserts the loop $r_{1,j}^{\pm1}$ into the new attaching
circle of that 2-handle (with $\pm1$ depending on the orientation chosen).

By keeping track of the attaching 2-spheres of the trivial 3-handles after the
handle slides, it is possible to attach 3-handles $h_{2,1}^{3}$, $h_{2,2}^{3}%
$, and $h_{2,3}^{3}$ to $C_{1}\cup h_{2,1}^{2}\cup h_{2,2}^{2}\cup h_{2,3}%
^{2}$ that are complementary to $h_{2,1}^{2}$, $h_{2,2}^{2}$, and $h_{2,3}%
^{2}$, respectively. Then
\[
C_{1}\cup\left(  \cup_{j=1}^{3}h_{2,j}^{2}\right)  \cup\left(  \cup_{j=1}%
^{3}h_{2,j}^{3}\right)  \approx\medskip C_{1}.
\]
For later purposes, it is useful to have a schematic image of the attaching
circles of $\{h_{1,j}^{2}\}_{j=1}^{3}$ and the attaching $2$-spheres of the
complementary handles $\{h_{1,j}^{2}\}_{j=1}^{3}$. Figure
\ref{Fig4-AttachingCircle} provides such an image for one complementary pair.%
\begin{figure}
[ptb]
\begin{center}
\includegraphics[
height=2.5374in,
width=2.5374in
]%
{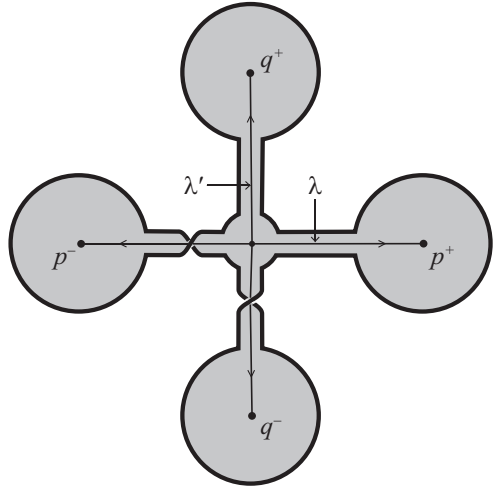}%
\caption{Attaching a $\left(  2,3\right)  $-handle pair}%
\label{Fig4-AttachingCircle}%
\end{center}
\end{figure}

\noindent The outer loop represents the attaching circle for an $h_{2,j}^{2}$
and the shaded region represents the `lower hemisphere' of the attaching
2-sphere of $h_{2,j}^{3}$; the `upper hemisphere', which is not shown, is a
parallel copy of the core of $h_{2,j}^{2}$. Within the lower hemisphere, the
small central disk represents the lower hemisphere of the 2-sphere before
handle slides. The arms are narrow strips whose centerlines are the paths
along which the handle slides were performed; diametrically opposite paths
lead to the same 2-handle, and are chosen to be parallel to a fixed path. We
have indicated this by labeling one pair of centerlines $\lambda$ and the
other $\lambda^{\prime}$. The four outer disks are parallel to the cores of
the 2-handles over which the slides were made. A twist in the strip leading to
an outer disk is used to reverse the orientation of the boundary of that disk.
Thus, diametrically opposite outer disks are parallel to each other, but with
opposite orientations. Center points of the outer disks represent transverse
intersections with belt spheres of those handles; thus, $p^{+}$ and $p^{-}$
are nearby intersections with the same belt sphere, and similarly for $q^{+}$
and $q^{-}$.

By rewriting $C_{1}\cup\left(  \cup_{j=1}^{3}h_{2,j}^{2}\right)  \cup\left(
\cup_{j=1}^{3}h_{2,j}^{3}\right)  $ as $C_{0}\cup\left(  \cup_{j=1}^{3}%
h_{1,j}^{2}\right)  \cup\left(  \cup_{j=1}^{3}h_{2,j}^{2}\right)  \cup\left(
\cup_{j=1}^{3}h_{2,j}^{3}\right)  $, we may reorder the handles so that
$h_{2,1}^{2}$, $h_{2,2}^{2}$, and $h_{2,3}^{2}$ are attached first. Define
\[
C_{2}=C_{0}\cup\left(  \cup_{j=1}^{3}h_{2,j}^{2}\right)
\]
and note that $\pi_{1}\left(  C_{2}\right)  \approx\pi_{1}\left(  \partial
C_{2}\right)  \approx G_{2}$. Furthermore,
\[
C_{2}\cup\left(  \cup_{j=1}^{3}h_{1,j}^{2}\right)  \cup\left(  \cup_{j=1}%
^{3}h_{2,j}^{3}\right)  \approx C_{1}.
\]
So, if we let
\[
A_{1}=(\partial C_{2}\times\left[  0,\varepsilon\right]  )\cup\left(
\cup_{j=1}^{3}h_{1,j}^{2}\right)  \cup\left(  \cup_{j=1}^{3}h_{2,j}%
^{3}\right)  \text{,}%
\]
(the result of excising the interior of a slightly shrunken copy of $C_{2}$),
then $\partial A_{1}\approx\partial C_{2}\sqcup\partial C_{1}$. By letting
$\Gamma_{1}=\partial C_{1}$ and $\Gamma_{2}=\partial C_{2}$ we obtain the
first cobordism of the construction $\left(  A_{1},\Gamma_{1},\Gamma
_{2}\right)  $. By avoiding the base point $p_{0}\in\partial C_{0}$ in all of
the above handle additions, we may let the arc $\rho_{1}\subseteq A_{1}$ be
the product line $p_{0}\times\left[  0,\varepsilon\right]  $, with $p_{1}$ and
$p_{2}$ its end points. Conditions (a) and (b) are then clear.\medskip

\noindent\textbf{Inductive Step. }(Constructing $A_{i}$ and $C_{i+1}$)

Assume the existence of a cobordism $\left(  A_{i-1},\Gamma_{i-1},\Gamma
_{i}\right)  $ satisfying (a) and (b) along with a compact manifold
$C_{i}=C_{0}\cup\left(  \cup_{j=1}^{3}h_{i,j}^{2}\right)  $, with the
attaching circle of each $h_{i,j}^{2}$ representing the relator $r_{i,j}$ in
the presentation of $G_{i}$, and $\partial C_{i}=\Gamma_{i}$. Attach three
trivial $\left(  2,3\right)  $-handle pairs to $C_{i}$, then perform handle
slides on each of the trivial 2-handles (over the handles $\{h_{i,j}%
^{2}\}_{j=1}^{3}$) so that the resulting 2-handles $h_{i+1,1}^{2}$,
$h_{i+1,2}^{2}$ and $h_{i+1,3}^{2}$ have attaching circles spelling out the
words $r_{i+1,1}$, $r_{i+1,2}$ and $r_{i1,3}$, respectively. This is possible
since each $r_{i+1,k}$ can be viewed as a product of the loops $\left\{
r_{i,j}\right\}  _{j=1}^{3}$ and their inverses, which are the attaching
circles of $\{h_{i,j}^{2}\}_{j=1}^{3}$.

By keeping track of the attaching 2-spheres of the trivial 3-handles under the
above handle slides, it is possible to attach 3-handles $h_{i+1,1}^{3}$,
$h_{i+1,2}^{3}$, and $h_{i+1,3}^{3}$ to $C_{i}\cup h_{i+1,1}^{2}\cup
h_{i+1,2}^{2}\cup h_{i+1,3}^{2}$ that are complementary to $h_{i+1,1}^{2}$,
$h_{i+1,2}^{2}$, and $h_{i+1,3}^{2}$, respectively. Then
\[
C_{i}\cup\left(  \cup_{j=1}^{3}h_{i+1,j}^{2}\right)  \cup\left(  \cup
_{j=1}^{3}h_{i+1,j}^{3}\right)  \approx C_{i}.
\]
A picture like Figure \ref{Fig4-AttachingCircle}, but with different indices,
describes the current situation.

Rewrite $C_{i}\cup\left(  \cup_{j=1}^{3}h_{i+1,j}^{2}\right)  \cup\left(
\cup_{j=1}^{3}h_{i+1,j}^{3}\right)  $ as $C_{0}\cup\left(  \cup_{j=1}%
^{3}h_{i,j}^{2}\right)  \cup\left(  \cup_{j=1}^{3}h_{i+1,j}^{2}\right)
\cup\left(  \cup_{j=1}^{3}h_{i+1,j}^{3}\right)  $, then reorder the handles so
that $h_{i+1,1}^{2}$, $h_{i+1,2}^{2}$, and $h_{i+1,3}^{2}$ are attached first.
Define
\[
C_{i+1}=C_{0}\cup\left(  \cup_{j=1}^{3}h_{i+1,j}^{2}\right)
\]
and note that $\pi_{1}\left(  C_{i+1}\right)  \approx\pi_{1}\left(  \partial
C_{i+1}\right)  \approx G_{i+1}$.

Furthermore,
\[
C_{i+1}\cup\left(  \cup_{j=1}^{3}h_{i,j}^{2}\right)  \cup\left(  \cup
_{j=1}^{3}h_{i+1,j}^{3}\right)  \approx C_{i}.
\]
Excising the interior of a slightly shrunken copy of $C_{i+1}$ gives%
\[
A_{i+1}=(\partial C_{i+1}\times\left[  0,\varepsilon\right]  )\cup\left(
\cup_{j=1}^{3}h_{i,j}^{2}\right)  \cup\left(  \cup_{j=1}^{3}h_{i+1,j}%
^{3}\right)  \text{,}%
\]
then $\partial A_{i+1}\approx\partial C_{i+1}\sqcup\partial C_{i}$. Noting
that $\Gamma_{i}=\partial C_{i}$ and letting $\Gamma_{i+1}=\partial C_{i+1}$,
we obtain $\left(  A_{i},\Gamma_{i},\Gamma_{i+1}\right)  $. By avoiding
$p_{i}\in\partial C_{i}$ in all of the handle additions, letting $\rho
_{i}\subseteq A_{i}$ be the product line $p_{i}\times\left[  0,\varepsilon
\right]  $, and $p_{i+1}$ the new end point, conditions (a) and (b) are
clear.\medskip

Assembling the pieces in the manner described in Figure
\ref{Fig3-Union of cobordisms} completes the construction. In particular, we
obtain a 1-ended open manifold%
\[
M^{n}=C_{1}\cup A_{1}\cup A_{2}\cup A_{3}\cup\cdots
\]
whose fundamental group at infinity is represented by the inverse sequence
(\ref{Chosen inverse sequence}).

\begin{remark}
\label{Remark: Dual handle decompositions}\emph{In the construction of
}$\left(  A_{i},\Gamma_{i},\Gamma_{i+1}\right)  $\emph{, we have written
}$\Gamma_{i}$\emph{ on the left and }$\Gamma_{i+1}$\emph{ on the right to
match the blueprint laid out in Figure \ref{Fig3-Union of cobordisms}. In that
case, the handle decomposition of }$A_{i}$\emph{ implicit in the construction
goes from right to left, with handles being attached to a collar neighborhood
}$\Gamma_{i+1}\times\left[  0,\varepsilon\right]  $\emph{ of }$\Gamma_{i+1}%
$\emph{. Later, when our perspective becomes reversed, we will pass to the
dual decomposition }%
\[
A_{i}=(\Gamma_{i}\times\left[  0,\varepsilon\right]  )\cup\left(  \cup
_{j=1}^{3}\overline{h}_{1,j}^{n-3}\right)  \cup\left(  \cup_{j=1}^{3}%
\overline{h}_{2,j}^{n-2}\right)
\]
\emph{where each }$\overline{h}^{n-p}$\emph{ is the dual of an original
}$h^{p}$\emph{ and }$\Gamma_{i}\times\left[  0,\varepsilon\right]  $\emph{ is
a thin collar neighborhood of }$\Gamma_{i}$\emph{.}
\end{remark}

\subsubsection{Absolute inward tameness of $M^{n}$}

The following proposition will complete the proof of Theorem
\ref{Theorem: Existence of counterexamples}.

\begin{proposition}
\label{Prop: absolute inward tameness of Mn}For the manifolds $M^{n}$
constructed above, each clean neighborhood of infinity
\[
N_{i}=A_{i}\cup A_{i+1}\cup A_{i+2}\cup\cdots
\]
has finite homotopy type. Thus, $M^{n}$ is absolutely inward tame.
\end{proposition}

This will be accomplished by examining $H_{\ast}\left(  N_{i},\Gamma_{i};%
\mathbb{Z}
G_{i}\right)  $ (equivalently, $H_{\ast}\left(  \widetilde{N}_{i}%
,\widetilde{\Gamma}_{i};%
\mathbb{Z}
\right)  $ viewed as a $%
\mathbb{Z}
G_{i}$-module) where $G_{i}=\pi_{1}\left(  N_{i}\right)  =\pi_{1}\left(
\Gamma_{i}\right)  $. In particular, we will prove:

\begin{claim}
\label{Claim: homology of N_i}For each $i$, $H_{\ast}\left(  N_{i},\Gamma_{i};%
\mathbb{Z}
G_{i}\right)  $ is trivial in all dimensions except for $\ast=n-2$, where it
is isomorphic to the free module $\left(
\mathbb{Z}
G_{i}\right)  ^{3}=%
\mathbb{Z}
G_{i}\oplus%
\mathbb{Z}
G_{i}\oplus%
\mathbb{Z}
G_{i}$.
\end{claim}

Once this claim is established, Proposition
\ref{Prop: absolute inward tameness of Mn} follows from \cite[Lemma 6.2]{Si}.
In Remark \ref{Siebenmann's Lemma} at the conclusion of this section, we
explain why this final observation is elementary, requiring no discussion of
finite dominations or finiteness obstructions.

For proving the claim, it is useful to consider compact subsets of the form%
\[
A_{i,k}=A_{i}\cup A_{i+1}\cup\cdots\cup A_{k}.
\]
By repeated application of Remark \ref{Remark: Dual handle decompositions},
there is a handle decomposition of $A_{i,k}$ based on $\Gamma_{i}\times\left[
0,\varepsilon\right]  $ with handles only of indices $n-3$ and $n-2$. By
reordering the handles, $\left(  A_{i,k},\Gamma_{i}\right)  $ is seen to be
homotopy equivalent to a finite relative CW complex $\left(  K_{i,k}%
,\Gamma_{i}\right)  $ where $K_{i,k}$ consists of $\Gamma_{i}$ with an
$\left(  n-3\right)  $-cell attached for each $\left(  n-3\right)  $-handle of
$A_{i,k}$ followed by an $\left(  n-2\right)  $-cell for each $\left(
n-2\right)  $-handle. In the usual way, the $%
\mathbb{Z}
G_{i}$-incidence number of an $\left(  n-2\right)  $-cell with an $\left(
n-3\right)  $-cell is equal to the $%
\mathbb{Z}
G_{i}$-intersection number between the belt sphere of the corresponding
$\left(  n-3\right)  $-handle and the attaching sphere of the corresponding
$\left(  n-2\right)  $-handle. This process produces a sequence%
\[
K_{i,i}\subseteq K_{i,i+1}\subseteq K_{i,i+2}\subseteq\cdots
\]
of relative CW complexes with direct limit a relative CW pair $\left(
K_{i,\infty},\Gamma_{i}\right)  $ homotopy equivalent to $\left(  N_{i}%
,\Gamma_{i}\right)  $. So we can determine $H_{\ast}\left(  N_{i},\Gamma_{i};%
\mathbb{Z}
G_{i}\right)  $ by calculating $H_{\ast}\left(  A_{i,k},\Gamma_{i};%
\mathbb{Z}
G_{i}\right)  $ and taking the direct limit as $k\rightarrow\infty.$

The $%
\mathbb{Z}
G_{i}$-handle chain complex for $\left(  A_{i,k},\Gamma_{i}\right)  $
(equivalently, the $%
\mathbb{Z}
G_{i}$-cellular chain complex for $\left(  K_{i,k},\Gamma_{i}\right)  $) looks
like%
\[
0\longrightarrow\mathcal{C}_{n-2}\overset{\partial}{\longrightarrow
}\mathcal{C}_{n-3}\longrightarrow0
\]
where $\mathcal{C}_{n-2}$ and $\mathcal{C}_{n-3}$ are finitely generated free
$%
\mathbb{Z}
G_{i}$-modules generated by the handles of $A_{i,k}$, and the boundary map is
determined by $%
\mathbb{Z}
G_{i}$-intersection numbers between the belt spheres of $\left(  n-3\right)
$-handles and attaching spheres of the $\left(  n-2\right)  $-handles. These
intersection numbers will be determined by returning to the
construction.\medskip

Beginning with the compact manifold $C_{i}=C_{0}\cup\left(  \cup_{j=1}%
^{3}h_{i,j}^{2}\right)  $, attach three trivial $\left(  2,3\right)  $-handle
pairs, then perform handle slides on the 2-handles (over the handles
$\{h_{i,j}^{2}\}_{j=1}^{3}$) to obtain $h_{i+1,1}^{2}$, $h_{i+1,2}^{2}$ and
$h_{i+1,3}^{2}$ with attaching circles $r_{i+1,1}$, $r_{i+1,2}$ and
$r_{i+1,3}$, respectively. Having kept track of the attaching 2-spheres of the
trivial 3-handles under the handle slides, attach 3-handles $h_{i+1,1}^{3}$,
$h_{i+1,2}^{3}$, and $h_{i+1,3}^{3}$ to $C_{i}\cup h_{i+1,1}^{2}\cup
h_{i+1,2}^{2}\cup h_{i+1,3}^{2}$ that are complementary to $h_{i+1,1}^{2}$,
$h_{i+1,2}^{2}$, and $h_{i+1,3}^{2}$, respectively (all as described
`inductive step' above). This can all be done so that $h_{i+1,1}^{3}$,
$h_{i+1,2}^{3}$, and $h_{i+1,3}^{3}$ do not touch the earlier $2$-handles
$h_{i,1}^{2}$, $h_{i,2}^{2}$ and $h_{i,3}^{2}$. Next attach a second trio of
trivial $\left(  2,3\right)  $-handle pairs, taking care that they are
disjoint from the existing handles, and slide the trivial $2$-handles over the
$2$-handles $\{h_{i+1,j}^{2}\}_{j=1}^{3}$ so that the resulting 2-handles
$\{h_{i+2,j}^{2}\}_{j=1}^{3}$ have attaching circles $r_{i+2,1}$, $r_{i+2,2}$
and $r_{i+2,3}$. Again, having kept track of the attaching 2-spheres of the
trivial 3-handles under the handle slides, attach 3-handles $h_{i+2,1}^{3}$,
$h_{i+2,2}^{3}$, and $h_{i+2,3}^{3}$ to%
\[
C_{i}\cup\left(  \cup_{j=1}^{3}h_{i+1,j}^{2}\right)  \cup\left(  \cup
_{j=1}^{3}h_{i+1,j}^{3}\right)  \cup\left(  \cup_{j=1}^{3}h_{i+2,j}%
^{2}\right)
\]
that are complementary to $h_{i+2,1}^{2}$, $h_{i+2,2}^{2}$, and $h_{i+2,3}%
^{2}$, respectively, while taking care that these new 3-handles are completely
disjoint from all 2- and 3-handles of lower index. Continue this process $k-i$
times, at each stage: attaching three trivial $\left(  2,3\right)  $-handle
pairs disjoint from the existing handles; sliding the trivial $2$-handles over
the $2$-handles created in the previous step, in the manner prescribed above;
then attaching 3-handles complementary to these new 2-handles (and disjoint
from earlier 2- and 3-handles) along the images of the attaching 2-spheres of
the trivial 3-handles after the handle slides.

Since all of the 2- and 3-handles mentioned above, except for the original
2-handles $h_{i,1}^{2}$, $h_{i,2}^{2}$ and $h_{i,3}^{2}$, occur in
complementary pairs, the manifold we just created is just a thickened copy of
$C_{i}$; let us call it $C_{i}^{\prime}$. By the standard reordering lemma, we
may arrange that the 2-handles are pairwise disjoint, and all are attached
before any of the 3-handles---which are also are attached in a pairwise
disjoint manner. Then%

\begin{align*}
C_{i}^{\prime}  &  =C_{i}\cup\left(  \cup_{s=1}^{k}\left(  \cup_{j=1}%
^{3}h_{i+s,j}^{2}\right)  \right)  \cup\left(  \cup_{s=1}^{k}\left(
\cup_{j=1}^{3}h_{i+s,j}^{3}\right)  \right) \\
&  =C_{0}\cup\left(  \cup_{j=1}^{3}h_{i,j}^{2}\right)  \cup\left(  \cup
_{s=1}^{k}\left(  \cup_{j=1}^{3}h_{i+s,j}^{2}\right)  \right)  \cup\left(
\cup_{s=1}^{k}\left(  \cup_{j=1}^{3}h_{i+s,j}^{3}\right)  \right) \\
&  =C_{0}\cup\left(  \cup_{j=1}^{3}h_{i+k,j}^{2}\right)  \cup\left(
\cup_{j=1}^{3}h_{i,j}^{2}\right)  \cup\left(  \cup_{s=1}^{k-1}\left(
\cup_{j=1}^{3}h_{i+s,j}^{2}\right)  \right)  \cup\left(  \cup_{s=1}^{k}\left(
\cup_{j=1}^{3}h_{i+s,j}^{3}\right)  \right) \\
&  =C_{k}\cup\left(  \cup_{j=1}^{3}h_{i,j}^{2}\right)  \cup\left(  \cup
_{s=1}^{k-1}\left(  \cup_{j=1}^{3}h_{i+s,j}^{2}\right)  \right)  \cup\left(
\cup_{s=1}^{k}\left(  \cup_{j=1}^{3}h_{i+s,j}^{3}\right)  \right)
\end{align*}
where, going from the first to the second line, we apply the definition of
$C_{i}$; going from the second to the third, we bring the last triple of
2-handles forward to the beginning; and in going from the third to the fourth,
we apply the definition of $C_{k}$.

Excising a slightly shrunken copy of the interior of $C_{k}$ from
$C_{i}^{\prime}$ results in a cobordism between $\partial C_{k}=\Gamma_{k}$
and $\partial C_{i}^{\prime}\approx\Gamma_{i}$, which has a handle
decomposition%
\[
(\Gamma_{k}\times\left[  0,\varepsilon\right]  )\cup\left(  \cup_{j=1}%
^{3}h_{i,j}^{2}\right)  \cup\left(  \cup_{s=1}^{k-1}\left(  \cup_{j=1}%
^{3}h_{i+s,j}^{2}\right)  \right)  \cup\left(  \cup_{s=1}^{k}\left(
\cup_{j=1}^{3}h_{i+s,j}^{3}\right)  \right)  \text{.}%
\]
Comparing this handle decomposition to our earlier construction, reveals that
this cobordism is precisely $A_{i}\cup A_{i+1}\cup\cdots\cup A_{k}=A_{i,k}$.
In order to match the orientation of Figure \ref{Fig3-Union of cobordisms},
view $\Gamma_{k}$ as the right-hand boundary and $\Gamma_{i}$ as the left-hand
boundary, with 2- and 3-handles being attached from right to left. Before
switching to the dual handle decomposition, we analyze the $%
\mathbb{Z}
G_{i}$-intersection numbers between the attaching spheres of the 3-handles and
the belt spheres of the 2-handles. All should be viewed as submanifolds of the
left-hand boundary of $(\Gamma_{k}\times\left[  0,\varepsilon\right]
)\cup\left(  \cup_{j=1}^{3}h_{i,j}^{2}\right)  \cup\left(  \cup_{s=1}%
^{k-1}\left(  \cup_{j=1}^{3}h_{i+s,j}^{2}\right)  \right)  $, which has
fundamental group $G_{i}$.

For each $1\leq s\leq k$ and $j\in\left\{  1,2,3\right\}  $ let $\alpha
_{i+s,j}^{2}$ denote the attaching 2-sphere of $h_{i+s,j}^{3}$; and for each
$0\leq s^{\prime}\leq k-1$ and $j^{\prime}\in\left\{  1,2,3\right\}  $ let
$\beta_{i+s^{\prime},j^{\prime}}^{n-3}$ denote the belt $\left(  n-3\right)
$-sphere of $h_{i+s^{\prime},j^{\prime}}^{2}$ There are three cases to
consider.\medskip

\noindent\textbf{Case 1.} $s=s^{\prime}$.

Then for each $j$, the pair $\left(  h_{i+s,j}^{2},h_{i+s,j}^{3}\right)  $ is
complementary; in other words $\alpha_{i+s,j}^{2}$ intersects $\beta
_{i+s,j}^{n-3}$ transversely in a single point. Adjusting base paths, if
necessary, and being indifferent to orientation (since it will not affect our
computations), we have $\varepsilon_{%
\mathbb{Z}
G_{i}}\left(  \alpha_{i+s,j}^{2},\beta_{i+s,j}^{n-3}\right)  =\pm1$. If $j\neq
j^{\prime}$, then $h_{i+s,j}^{3}$ does not intersect $h_{i+s,j^{\prime}}^{2}$,
so $\varepsilon_{%
\mathbb{Z}
G_{i}}\left(  \alpha_{i+s,j}^{2},\beta_{i+s,j^{\prime}}^{n-3}\right)
=0$.\medskip

\noindent\textbf{Case 2.} $s=s^{\prime}+1$.

For each $j$, $\alpha_{i+s,j}^{2}$ can be split into a pair of disks. The
`upper hemisphere' lies in the the 2-handle $h_{i+s,j}^{2}$ and intersects
$\beta_{i+s,j}^{n-3}$ transversely in a single point; that point of
intersection was accounted for in Case 1. The `lower hemisphere' is analogous
to the one pictured in Figure \ref{Fig4-AttachingCircle}. If $\left\{
u,v\right\}  =\left\{  1,2,3\right\}  -\left\{  j\right\}  $, then one pair of
the diametrically opposite disks has boundaries labelled $r_{i+s-1,u}$ and
$r_{i+s-1,u}^{-1}$ and the disks are parallel to the core of $h_{i+s-1,u}^{2}%
$, so each intersects $\beta_{i+s-1,u}^{n-3}$ transversely in points
$p_{u}^{+}$ and $p_{u}^{-}$. Due to the flipped orientation of one of the
disks, these points of intersection, between $\alpha_{i+s,j}^{2}$ and
$\beta_{i+s-1,u}^{n-3}$, have opposite sign. Connecting $p_{u}^{+}$ and
$p_{u}^{-}$ by a path homotopic to $\lambda^{-1}\ast\lambda$ in $\alpha
_{i+s,j}^{2}$ and a short path $\mu$ connecting $p_{u}^{+}$ and $p_{u}^{-}$ in
$\beta_{i+s-1,u}^{n-3}$ yields a loop that is contractible in the left-hand
boundary of $(\Gamma_{k}\times\left[  0,\varepsilon\right]  )\cup\left(
\cup_{j=1}^{3}h_{i,j}^{2}\right)  \cup\left(  \cup_{s=1}^{k-1}\left(
\cup_{j=1}^{3}h_{i+s,j}^{2}\right)  \right)  $. So together $p_{u}^{+}$ and
$p_{u}^{-}$ contribute $0$ to to the $%
\mathbb{Z}
G_{i}$-intersection number of $\alpha_{i+s,j}^{2}$ and $\beta_{i+s-1,u}^{n-3}%
$; hence, $\varepsilon_{%
\mathbb{Z}
G_{i}}\left(  \alpha_{i+s,j}^{2},\beta_{i+s-1,u}^{n-3}\right)  =0$. Similarly
$\varepsilon_{%
\mathbb{Z}
G_{i}}\left(  \alpha_{i+s,j}^{2},\beta_{i+s-1,v}^{n-3}\right)  =0$. Finally,
$\alpha_{i+s,j}^{2}$ and $\beta_{i+s-1,j}^{n-3}$ do not intersect, so
$\varepsilon_{%
\mathbb{Z}
G_{i}}\left(  \alpha_{i+s,j}^{2},\beta_{i+s-1,j}^{n-3}\right)  =0$, as
well.\medskip

\noindent\textbf{Case 3.} $s\notin\{s^{\prime},s^{\prime}+1\}$.

In this case, the handles $h_{i+s,j}^{3}$ and $h_{i+s^{\prime},u}^{2}$ are
disjoint, so $\varepsilon_{%
\mathbb{Z}
G_{i}}\left(  \alpha_{i+s,j}^{2},\beta_{i+s,j^{\prime}}^{n-3}\right)  =0$.
\medskip

Now invert the above handle decomposition, to obtain a handle decomposition of
the cobordism $\left(  A_{i,k},\Gamma_{i},\Gamma_{k}\right)  $, based on
$\Gamma_{i},$ containing only $\left(  n-3\right)  $- and $\left(  n-2\right)
$-handles. Specifically, we have%
\[
(\Gamma_{i}\times\left[  0,\varepsilon\right]  )\cup\left(  \cup_{s=1}%
^{k}\left(  \cup_{j=1}^{3}\overline{h}_{i+s,j}^{n-3}\right)  \right)
\cup\left(  \cup_{j=1}^{3}\overline{h}_{i,j}^{n-2}\right)  \cup\left(
\cup_{s=1}^{k-1}\left(  \cup_{j=1}^{3}\overline{h}_{i+s,j}^{n-2}\right)
\right)  \text{.}%
\]
Since the belt sphere of each $\overline{h}^{n-3}$ is the attaching sphere of
its dual $h^{3}$ and the attaching sphere of of each $\overline{h}^{n-2}$ is
the belt sphere of its dual $h^{2}$, the incidence numbers between these
handles of this handle decomposition are determined (up to sign) by the
earlier calculations. So the cellular $%
\mathbb{Z}
G_{i}$-chain complex for the $\left(  A_{i,k},\Gamma_{i}\right)  $ is
isomorphic to
\[
0\rightarrow\bigoplus_{s=0}^{k-1}\left(
\mathbb{Z}
G_{i}\right)  ^{3}\overset{\partial}{\longrightarrow}\bigoplus_{s=1}%
^{k}\left(
\mathbb{Z}
G_{i}\right)  ^{3}\rightarrow0
\]
where, the $\left(
\mathbb{Z}
G_{i}\right)  ^{3}$ summands on the left are generated by the handles
$\left\{  \overline{h}_{i+s,j}^{n-2}\right\}  _{j=1}^{3}$ and those on the
right by $\left\{  \overline{h}_{i+s,j}^{n-3}\right\}  _{j=1}^{3}$. Since
$\varepsilon_{%
\mathbb{Z}
G_{i}}\left(  \alpha_{i+s,j}^{2},\beta_{i+s,j}^{n-3}\right)  =\pm1$ for all
$1\leq s\leq k-1$ and all other intersection numbers are $0$, the boundary map
is trivial on the $0^{\text{th}}$ copy of $\left(
\mathbb{Z}
G_{i}\right)  ^{3}$; misses the $k^{\text{th}}$ copy of $\left(
\mathbb{Z}
G_{i}\right)  ^{3}$ in the range; and restricts to an isomorphism
$\bigoplus_{s=1}^{k-1}\left(
\mathbb{Z}
G_{i}\right)  ^{3}\overset{\cong}{\longrightarrow}\bigoplus_{s=1}^{k-1}\left(
%
\mathbb{Z}
G_{i}\right)  ^{3}$ elsewhere. Thus
\begin{align*}
H_{n-2}\left(  A_{i,k},\Gamma_{i};%
\mathbb{Z}
G_{i}\right)   &  =\ker\partial\cong\left(
\mathbb{Z}
G_{i}\right)  ^{3}\text{, and }\\
H_{n-3}\left(  A_{i,k},\Gamma_{i};%
\mathbb{Z}
G_{i}\right)   &  =\operatorname*{coker}\partial\cong\left(
\mathbb{Z}
G_{i}\right)  ^{3}%
\end{align*}
where $H_{n-2}\left(  K_{i,k},\Gamma_{i}\right)  $ is generated by the $s=0$
summand, and $H_{n-3}\left(  K_{i,k},\Gamma_{i}\right)  $ is generated by the
$s=k$ summand.

Now consider the inclusion $A_{i,k}\hookrightarrow A_{i,k+1}$ and the
corresponding inclusion of $%
\mathbb{Z}
G_{i}$-chain complexes. The chain complex of $A_{i,k+1}$ will contain an extra
$\left(
\mathbb{Z}
G_{i}\right)  ^{3}$ summand in each dimension, generated by $\left\{
\overline{h}_{i+k,j}^{n-2}\right\}  _{j=1}^{3}$ and $\left\{  \overline
{h}_{i+k+1,j}^{n-3}\right\}  _{j=1}^{3}$, respectively. The boundary map takes
the new summand in the domain onto the previous cokernel, thereby killing
$H_{n-3}\left(  A_{i,k},\Gamma_{i};%
\mathbb{Z}
G_{i}\right)  $, and replacing it with a cokernel generated by $\left\{
\overline{h}_{i+k+1,j}^{n-3}\right\}  _{j=1}^{3}$. Said differently, the
inclusion induced map%
\[
i_{\ast}:H_{n-3}\left(  K_{i,k},\Gamma_{i};%
\mathbb{Z}
G_{i}\right)  \overset{0}{\longrightarrow}H_{n-3}\left(  K_{i,k+1},\Gamma_{i};%
\mathbb{Z}
G_{i}\right)
\]
is trivial. On the other hand, the expansion from $K_{i,k}$ to $K_{i,k+1}$
does not change $\ker\partial$, which is still generated by the handles
$\left\{  \overline{h}_{i,j}^{n-2}\right\}  _{j=1}^{3}$. In other words, the
inclusion induced map
\[
i_{\ast}:H_{n-2}\left(  K_{i,k},\Gamma_{i};%
\mathbb{Z}
G_{i}\right)  \overset{\cong}{\longrightarrow}H_{n-2}\left(  K_{i,k+1}%
,\Gamma_{i};%
\mathbb{Z}
G_{i}\right)
\]
is an isomorphism.

Taking direct limits, we have
\[
H_{\ast}\left(  N_{i},\Gamma_{i};%
\mathbb{Z}
G_{i}\right)  \cong\left\{
\begin{array}
[c]{ccc}%
\left(
\mathbb{Z}
G_{i}\right)  ^{3} &  & \text{if }\ast\ =n-2\\
0 &  & \text{otherwise}%
\end{array}
\right.  \text{.}%
\]
So the claim is proved.

\begin{remark}
\emph{\label{Siebenmann's Lemma}The appeal to \cite[Lemma 6.2]{Si} may give
the impression that obtaining Proposition
\ref{Prop: absolute inward tameness of Mn} from Claim
\ref{Claim: homology of N_i} is complicated---that is not the case. The
conclusion can be obtained directly as follows: If }$\left\{  e_{i,j}%
^{n-2}\right\}  _{j=1}^{3}$\emph{ represents the cores of the }$\left(
n-2\right)  $\emph{-handles }$\left\{  \overline{h}_{i,j}^{n-2}\right\}
$\emph{, which generate }$H_{\ast}\left(  N_{i},\Gamma_{i};%
\mathbb{Z}
G_{i}\right)  $\emph{, abstractly attach }$\left(  n-2\right)  $\emph{-disks
}$\left\{  f_{i,j}^{n-2}\right\}  _{j=1}^{3}$\emph{ to }$\Gamma_{i}$\emph{
along their boundaries. This does not affect fundamental groups, so by
excision, the pair }\newline$\left(  N_{i}\cup\left\{  f_{i,j}^{n-2}\right\}
_{j=1}^{3},\Gamma_{i}\cup\left\{  f_{i,j}^{n-2}\right\}  _{j=1}^{3}\right)
$\emph{ has the same }$%
\mathbb{Z}
G_{i}$\emph{-homology as }$\left(  N_{i},\Gamma_{i}\right)  $\emph{, with the
same generating set. Now attach an }$\left(  n-1\right)  $\emph{-cell }%
$g_{j}^{n-1}$\emph{ along each sphere }$e_{i,j}^{n-2}\cup f_{i,j}^{n-2}$\emph{
to obtain a pair }%
\[
\left(  N_{i}\cup\left\{  f_{i,j}^{n-2}\right\}  _{j=1}^{3}\cup\left\{
g_{i,j}^{n-2}\right\}  _{j=1}^{3},\Gamma_{i}\cup\left\{  f_{i,j}%
^{n-2}\right\}  _{j=1}^{3}\right)
\]
\emph{with trivial }$%
\mathbb{Z}
G_{i}$\emph{-homology in all dimensions. It follows that }%
\[
\Gamma_{i}\cup\left\{  f_{i,j}^{n-2}\right\}  _{j=1}^{3}\hookrightarrow
N_{i}\cup\left\{  f_{i,j}^{n-2}\right\}  _{j=1}^{3}\cup\left\{  g_{i,j}%
^{n-1}\right\}  _{j=1}^{3}\text{.}%
\]
\emph{is a homotopy equivalence. But notice that each }$g_{i,j}^{n-1}$\emph{
has a free face }$f_{i,j}^{n-2}$\emph{, so }$N_{i}\cup\left\{  f_{i,j}%
^{n-2}\right\}  _{j=1}^{3}\cup\left\{  g_{i,j}^{n-1}\right\}  _{j=1}^{3}%
$\emph{ collapses onto }$N_{i}$\emph{. Therefore, }$N_{i}$\emph{ is homotopy
equivalent to }$\Gamma_{i}\cup\left\{  f_{i,j}^{n-2}\right\}  _{j=1}^{3}%
$\emph{.}\bigskip
\end{remark}

\section{A remaining question\label{Section: Remaining Questions}}

In the introduction we commented that nearly pseudo-collarable manifolds admit
arbitrarily small clean neighborhoods of infinity $N$, containing codimension
0 submanifolds $A$ for which $A\hookrightarrow N$ is a homotopy equivalence.
Call such a pair $\left(  N,A\right)  $ a \emph{wide homotopy collar}. The
difference, of course, between a wide homotopy collar and a homotopy collar is
that, in the latter, the subspace is required to be a codimension 0
submanifold. The fact that nearly pseudo-collarable manifolds contain
arbitrarily small wide homotopy collars is immediate from the following easy lemma.

\begin{lemma}
Suppose $N^{\prime}$ is a ($\operatorname*{mod}J$)-homotopy collar
neighborhood of infinity in a manifold $M^{n}$ ($n\geq5)$, where $J$ is a
normally finitely generated subgroup of $\ker\left(  \pi_{1}\left(  N^{\prime
}\right)  \rightarrow\pi_{1}\left(  M^{n}\right)  \right)  $. Then $M^{n}$
contains a wide homotopy collar neighborhood of infinity $\left(  N,A\right)
$, where $N^{\prime}\subseteq N\subseteq M^{n}$.
\end{lemma}

\begin{proof}
Choose a finite collection of pairwise disjoint properly embedded 2-disks
$\left\{  D_{i}^{2}\right\}  _{i=1}^{k}$ in $\overline{M^{n}-N^{\prime
}\text{,}}$ with boundaries comprising a normal generating set for
$\ker\left(  \pi_{1}\left(  N^{\prime}\right)  \rightarrow\pi_{1}\left(
M^{n}\right)  \right)  $. Let $\left(  N,A\right)  $ be a regular neighborhood
pair for $\left(  N^{\prime}\cup\left(  \cup_{i=1}^{k}D_{i}^{2}\right)
,\partial N^{\prime}\cup\left(  \cup_{i=1}^{k}D_{i}^{2}\right)  \right)  $ and
apply Lemma \ref{Lemma: (modL) h.e.}.
\end{proof}

Examples constructed in this paper and in \cite{GT1} show that the existence
of arbitrarily small wide homotopy collars in a manifold $M^{n}$ does not
imply the existence of a pseudo-collar structure. The following seems likely
but, so far, we have been unable to find a proof.\medskip

\noindent\textbf{Question. }\emph{If a manifold with compact boundary, }%
$M^{n}$\emph{, contains arbitrarily small wide homotopy collar neighborhoods
of infinity, must }$M^{n}$\emph{ be nearly pseudo-collarable? }


\begin{thebibliography}{9999}                                                                                             %


\bibitem[BRS]{BRS}S. Buoncristiano, C.P. Rourke, and B.J. Sanderson, \emph{A
Geometric Approach to Homology Theory}, LMS Lecture Note Series, Cambridge
University Press, 1976.

\bibitem[CS]{CS}T.A. Chapman and L.C. Siebenmann, \emph{Finding a boundary for
a Hilbert cube manifold}, Acta Math. \textbf{137} (1976), 171-208.

\bibitem[Coh]{Coh}M.M. Cohen, \emph{A Course in Simple-homotopy Theory},
Graduate Texts in Mathematics, Springer-Verlag, New York, 1973.

\bibitem[DK]{DK}J.F. Davis and P. Kirk, \emph{Lecture Notes in Algebraic
Topology}, Graduate Studies in Mathematics, 35. American Mathematical Society,
Providence, RI, 2001. xvi+367 pp.

\bibitem[FQ]{FQ}M.H. Freedman and F. Quinn, \emph{Topology of 4-manifolds},
Princeton University Press, Princeton, New Jersey, 1990.

\bibitem[Ge]{Ge}R. Geoghegan, \emph{Topological Methods in Group Theory},
Graduate Texts in Mathematics, Springer-Verlag, New York, 1973.

\bibitem[Gu1]{Gu1}C.R. Guilbault, \emph{Manifolds with non-stable fundamental
groups at infinity}, Geometry and Topology \textbf{4 }(2000), 537-579.

\bibitem[Gu2]{Gu2}C.R. Guilbault, \emph{Compactifications of manifolds with
boundary}, in progress.

\bibitem[Gu3]{Gu3}C.R. Guilbault, \emph{Ends, shapes, and boundaries in
manifold topology and geometric group theory, }http://arxiv.org/abs/1210.6741.

\bibitem[GT1]{GT1}C.R. Guilbault and F.C. Tinsley, \emph{Manifolds with
non-stable fundamental groups at infinity, II}, Geometry and Topology
\textbf{7 }(2003), 255-286.

\bibitem[GT2]{GT2}C.R. Guilbault and F.C. Tinsley, \emph{Manifolds with
non-stable fundamental groups at infinity, III}, Geometry and Topology
\textbf{10} (2006), 541--556.

\bibitem[GT3]{GT3}C.R. Guilbault and F.C. Tinsley, \emph{Spherical alterations
of handles: embedding the manifold plus construction,} Algebr. Geom. Topol. 13
(2013), no. 1, 35--60.

\bibitem[Ha]{Ha}A. Hatcher, \emph{Algebraic Topology}, Cambridge University
Press, Cambridge, UK, 2002.

\bibitem[LS]{LS}R.C. Lyndon and P.E. Schupp, \emph{Combinatorial Group
Theory}, Classics in Mathematics, Springer-Verlag, Berlin Heidelberg New York, 1970.

\bibitem[Ma]{Ma}W.S. Massey, \emph{Algebraic Topology: An Introduction},
Graduate Texts in Mathematics, Springer-Verlag, New York, 1967.

\bibitem[RS1]{RS1}C.P. Rourke and B.J. Sanderson, \emph{Block bundles: II.
Transversality}, Ann. Math. 87 (1968) 256--278.

\bibitem[RS2]{RS}C.P. Rourke and B.J. Sanderson, \emph{Introduction to
Piecewise-Linear Topology}, Springer-Verlag, New York, 1982.

\bibitem[Si]{Si}L.C. Siebenmann, \emph{The obstruction to finding a boundary
for an open manifold of dimension greater than five}, Ph.D. thesis, Princeton
University, 1965.

\bibitem[Stal]{Stal}J. Stallings, \emph{Homology and central series of
groups}, J. of Algebra \textbf{2} (1965), 170-181.

\bibitem[Stam]{Stam}U. Stammbach, \emph{Anwendungen der Holomogietheorie der
Gruppen auf Zentralreihen und auf Invarianten von Pr\"{a}sentierungen}
(German), Math. Z. \textbf{94} (1966) 157--177.

\bibitem[Wa]{Wa}C.T.C. Wall, \emph{Finiteness conditions for CW complexes},
Ann. Math. \textbf{8} (1965), 55-69.
\end{thebibliography}
\end{document}